%% file: paper.tex
\documentclass[a4paper,12pt]{article}
\usepackage{amssymb}
\usepackage{amsmath}
\usepackage{rotating}
\usepackage[title]{appendix}

\newcommand{\JJoin}{\!\!\Join\!\!}

\newcommand{\inc}{{\rm \tt I}}

\def\<{\langle}
\def\>{\rangle}

\def\pmod#1{\allowbreak\mkern5mu(\mathrm{mod}\,\,#1)}
\newfont\recht{cmr12}

\def\cS{{\cal S}}

\def\cP{{\cal P}}
\def\cL{{\cal L}}

\def\N{{\Bbb N}}

\newtheorem{lemma}{Lemma}[section]

\newtheorem{theo}[lemma]{Theorem}

\newenvironment{proof}{\noindent {\sc Proof.}}{\hfill$\Box$

\bigskip

}

\newcommand{\alt}[1]{{\sf A}_{#1}}
\newcommand{\mat}[1]{{\sf M}_{#1}}
\newcommand{\sy}[1]{{\sf S}_{#1}}

\newcommand{\soc}{\operatorname{Soc}}
\newcommand{\sym}[1]{{\sf Sym}\,#1}
\newcommand{\Alt}[1]{{\sf Alt}\,#1}

\renewcommand{\wr}{\,{\sf wr}\,}
\newcommand{\hol}{{\sf Hol}\,}

\newcommand{\cent}[2]{\mathbb C_{#1}(#2)}
\newcommand{\norm}[2]{\mathbb N_{#1}(#2)}

\newcommand{\aut}[1]{{\sf Aut}\,{#1}}

\renewcommand{\leq}{\leqslant}
\renewcommand{\geq}{\geqslant}

\parskip\medskipamount
\begin{document}
\title{\bf Primitive flag-transitive generalized hexagons and octagons}
\author{{\sc Csaba Schneider}\\
{Informatics Research Laboratory}\\
Computer and Automation Research Institute\\
1518 Budapest Pf.\ 63 Hungary\\
Email: csaba.schneider@sztaki.hu\\
WWW: www.sztaki.hu/$\sim$schneider\\\\
and
\\\\
{\sc Hendrik Van Maldeghem}\\
Department of Pure Mathematics and Computer Algebra\\
Ghent University\\
Krijgslaan 281, S22
B-9000 Ghent, Belgium\\
Email: hvm@cage.ugent.be\\
WWW: cage.ugent.be/$\sim$hvm}
\date{\today}
\maketitle
\begin{abstract}
Suppose that an automorphism group $G$ acts flag-transitively on a finite generalized hexagon or octagon $\cS$, and
suppose that the action on both the point and line set is primitive. We show that $G$ is an almost simple group
of Lie type, that is, the socle of $G$ is a simple Chevalley group.
\end{abstract}

\section{Introduction}

The classification of all finite flag-transitive generalized polygons is a long-standing important open problem
in finite geometry. Generalized polygons, introduced by Tits in \cite{Tit:59}, are among the most notable and
prominent examples of discrete geometries, they have a lot of applications and are the building bricks of the
Tits buildings. The determination of all finite flag-transitive examples would have a great impact on many
problems, not in the least because of a significant weakening of the hypotheses of many results.  It is
generally considered as an ``NC-hard'' problem, where NC stands for ``No Classification (of finite simple
groups allowed)''. By a result of Feit and Higman \cite{Fei-Hig:64}, we must only consider generalized
triangles (which are the projective planes), generalized quadrangles, generalized hexagons and generalized
octagons. In each case there are nontrivial examples of finite flag transitive geometries, and it is believed
that we know all of them. The most far reaching results are known for the class of projective planes, where the
only counterexamples would have a sharply transitive group on the flags, and the number of points must be a
prime number, see \cite{Kan:83}. For generalized quadrangles, besides the well known classical (and dual
classical, in the terminology of \cite{Pay-Tha:84}) cases there are exactly two other examples, both arising
from transitive hyperovals in Desarguesian projective planes, namely of respective order 4 and 16 (and the
hyperovals are the regular one and the Lunelli-Sce hyperoval, respectively). Both quadrangles have an affine
representation, that is, their point set can be identified with the point set of a $3$-dimensional affine space,
and the line set is the union of some parallel classes of lines of that space (more precisely, those parallel
classes of lines that define the corresponding hyperoval in the plane at infinity). For generalized hexagons
and octagons, only the classical (Moufang) examples are known to exist, and the conjecture is that they are the
only flag-transitive ones (and some even conjecture that they are the only finite ones!). An affine
construction similar to the one above for quadrangles can never lead to a generalized hexagon or octagon, and
this observation easily leads to the nonexistence of generalized hexagons and octagons admitting a primitive
point-transitive group and whose O'Nan-Scott type is HA (see below for precise definitions).

This observation is the starting point of the present paper. Since the classification of finite flag-transitive
generalized polygons is NC-hard, we have to break the problem down to a point where we must start a
case-by-case study involving the different classes of finite simple groups. One celebrated method is the use of
the famous O'Nan-Scott Theorem. This theorem distinguishes several 
classes of primitive permutation groups, one
being the class HA above. Another class is the class AS, the Almost Simple case, and this class contains all
known examples of finite flag-transitive generalized hexagons and octagons. Ideally, one would like to get rid
of all O'Nan-Scott classes except for the class AS. The rest of the proof would then consist of going through
the list of finite simple groups and try to prove that the existing examples are the only possibilities. In the
present paper, we achieve this goal. We even do a little better and prove that we can restrict to Chevalley
groups, that is, we rule out the almost simple groups with alternating socle, the sporadic groups being eliminated
already in \cite{Bue-Mal}. The treatment of the different classes of Chevalley groups is a nontrivial but | so
it appears | a feasible job, and shall be pursued elsewhere. Note that the classical hexagons and octagons have a
flag-transitive automorphism group of almost simple type with socle the simple Chevalley groups of type
$\mathsf{G}_2$, $^3\mathsf{D}_4$ and $^2\mathsf{F}_4$. Their construction is
with the natural BN-pair. 
The automorphism group of these polygons is primitive on both 
the point-set and the line-set, and it is also flag-transitive.

We note, however, that our assumptions include primitive actions on both the 
point and the line set of the
generalized hexagon or octagon. 
In some case, this can be weakened, and we have stated our intermediate and
partial results each time under the weakest hypotheses. This could be important for future use when trying to
reduce the general case to the primitive one handled in large in this paper.

A similar treatment for the finite generalized quadrangles seem out of reach for the moment. Therefore, we
restrict ourselves to the cases of hexagons and octagons for the rest of the paper.

\section{Setting}

Let $\cS=(\cP,\cL,\inc)$ be a finite point-line geometry, where
$\cP$ is a point set, $\cL$ is a line set, and $\inc$ is a
binary symmetric incidence relation. The \emph{incidence graph} of
$\cS$ is the graph with vertex set $\cP\cup\cL$, where the
adjacency relation is given by the incidence relation $\inc$. The
\emph{diameter} of $\cS$ is by definition the diameter of the
incidence graph of $\cS$, and the \emph{gonality} of $\cS$ is by
definition half of the girth of the incidence graph of $\cS$
(which is a bipartite graph and therefore has even girth). For $n\geq 2$,
the geometry $\cS$ is a \emph{weak generalized $n$-gon},
if
both the diameter and the gonality of $\cS$ are equal to $n$. If
every point is incident with at least three lines, and every line
carries at least three points, then we say that $\cS$ is
\emph{thick}, and we call it a \emph{generalized $n$-gon}, or
\emph{generalized polygon}. In this case, there are positive
integers $s,\ t\geq 2$ such that every line is incident with $s+1$
points, and every point is incident with $t+1$ lines. We call
$(s,t)$ the \emph{order} of $\cS$. If $n=2$, then $\cS$ is a
trivial geometry where every point is incident with every line. If
$n=3$, then $\cS$ is a projective plane.

A generalized $6$-gon (or \emph{hexagon}) $\cS$ with order $(s,t)$ has $(1+s)(1+st+s^2t^2)$ points and
$(1+t)(1+st+s^2t^2)$ lines. The number of flags, that is the number of incident point-line pairs,  of $\cS$ is
equal to $(1+s)(1+t)(1+st+s^2t^2)$. Also, it is well known that $st$ is a
perfect square (see \cite{Fei-Hig:64,vm}).
A generalized $8$-gon (or \emph{octagon}) $\cS$ with order $(s,t)$ has $(1+s)(1+st)(1+s^2t^2)$ points and
$(1+t)(1+st)(1+s^2t^2)$ lines. The number of flags of $\cS$ is equal to $(1+s)(1+t)(1+st)(1+s^2t^2)$. Also, it
is well known that $2st$ is a perfect square (see \cite{Fei-Hig:64,vm}). 
Hence one of $s,\ t$ is even and
consequently, either the number of points or the number of lines of $\cS$ is odd.

Let $\cS=(\cP,\cL,\inc)$ be a generalized hexagon or octagon. A {\em collineation} or {\em automorphism} of $G$
is a permutation of the point set $\cP$, together with a permutation of the line set $\cL$, preserving
incidence. The group of automorphisms is denoted by $\aut\cS$ and is referred to as the {\em automorphism
group} of $\cS$.  If $G$ is a group of automorphisms of $\cS$, then $G$ can be viewed as a permutation group on
$\cP$ and also as a permutation group on $\cL$. The main theorem of this paper is the following.

\begin{theo}\label{simple}
Suppose that $G$ is a group of automorphisms of a generalized hexagon or
octagon $\cS=(\cP,\cL,\inc)$. If $G$ is primitive on  both $\cP$ and $\cL$ and
$G$ is flag-transitive then $G$ must be an almost simple group of Lie type.
\end{theo}

\section{Some preliminary results}

The next result will be useful to rule out the existence of generalized polygons with a certain number of
points. Suppose that $n$ is a natural number and suppose that $n=3^{\alpha}p_1^{\alpha_1}\cdots p_k^{\alpha_k}$
where the $p_i$ are pairwise distinct primes all different from 3, $\alpha\geq 0$ and $\alpha_i\geq 1$ for all
$i$. Then we define the following quantities:
\begin{eqnarray*}
a(n)&=&3^{\max\{0,\alpha-1\}}\prod_{p_i\not\equiv 1\bmod 3} p_i^{\alpha_i};\\
b(n)&=&\prod_{p_i\not\equiv 1\bmod 4} p_i^{\alpha_i}.
\end{eqnarray*}
We obtain the following result about the number of
points of a generalized hexagon or octagon.

\begin{lemma}\label{21}
Suppose that $\cS=(\cP,\cL,\inc)$ is a generalized hexagon or octagon.
\begin{itemize}
\item[(i)] If $\cS$ is a generalized hexagon, then $a(|\cP|)^3\leq
  |\cP|$.
\item[(ii)] If $\cS$ is a generalized octagon, then $b(|\cP|)^2\leq
  |\cP|$.
\end{itemize}
\end{lemma}
\begin{proof}
(i) Suppose that $\cS$ is a generalized hexagon with order $(s,t)$.
Then $|\cP|=(1+s)(1+st+s^2t^2)$. As
  mentioned in the previous section, $st$ is a square, and it was proved in the last paragraph
  of~\cite[page~90]{Bue-Mal} that if $p$ is a prime such that
  $p|1+st+s^2t^2$, then $p\equiv 1\pmod 3$; in addition, $1+st+s^2t^2$ is not
  divisible by $9$. Thus $a(|\cP|)$ must divide $1+s$ and $|\cP|/a(|\cP|)$
  must be divisible by $1+st+s^2t^2$.
On the other hand, since $t\geq 2$, we obtain that
$(1+s)^2\leq (1+st+s^2t^2)$, which implies that $a(|\cP|)^2\leq
  |\cP|/a(|\cP|)$, and so part~(i) is valid.

(ii)  Suppose that $\cS$ is a generalized octagon with order $(s,t)$.
Then $|\cP|=(1+s)(1+st)(1+s^2t^2)$. As
  mentioned above, $2st$ is a square, and it was proved in~\cite[page~99]{Bue-Mal} that, if $p$ is a prime such that
  $p|1+s^2t^2$, then $p\equiv 1\pmod 4$.
Thus $b(|\cP|)$ must divide $(1+s)(1+st)$ and $|\cP|/b(|\cP|)$
  must be divisible by $(1+s^2t^2)$.
On the other hand, since $s,\ t\geq 2$, it follows that
$(1+s)(1+st)\leq (1+s^2t^2)$, and so $b(|\cP|)\leq
  |\cP|/b(|\cP|)$, and statement~(ii) holds.
\end{proof}
We will use the following notation: if $x$ is a point collinear
with the point $y$, that is, $x$ and $y$ are incident with a
common line, then we write $x\sim y$. Dually, the notation
$L\sim M$ for lines $L,\ M$ means that $L$ and $M$ are concurrent;
that is,
they share a common point. If $x$ and $z$ are non collinear points
collinear to a common point $y$, then, assuming that the gonality is at least
$5$,  the point $y$ is unique with this
property and we write
$y=x\JJoin z$.

If $G$ is a permutation group acting on a set $\Omega$ then the image of $\omega\in\Omega$ under $g\in G$ is
denoted by $\omega g$, while the stabilizer in $G$ of $\omega$ is denoted by $G_\omega$. The group $G$ is said
to be {\em semiregular} if $G_\omega=1$ for all $\omega\in \Omega$, and it is said to be {\em regular} if it is
transitive and semiregular.

\begin{lemma}\label{lemma1}
If $\cS=(\cP,\cL,\inc)$ is a generalized hexagon or octagon with order $(s,t)$, then the
following is true.
\begin{itemize}
\item[(i)] If $\gcd(s,t)\neq 1$ and $g$ is an automorphism of $\cS$, then
either $g$ has a fixed point or there is a point $x\in\mathcal P$
such that $x\sim xg$.
\item[(ii)] If $\gcd(s,t)\neq 1$ and $g$ is an automorphism of $\cS$ with order
  $2$, then $g$ has either a fixed point or a fixed line.
In particular, if $G$ is an automorphism group of $\cS$ with even order, then
$G$ cannot be semiregular on both $\cP$ and $\cL$.
\item[(iii)] Let $x$ be a point and let
$y_1$ and $y_2$ be two points collinear with $x$
such that $y_1$ is not collinear with $y_2$. Suppose there are
automorphisms $g_1,\ g_2$ mapping $x$ to $y_1,\ y_2$, respectively. If
$g_1$ and $g_2$ commute, then $y_1{g_2}=y_2{g_1}=x$.
\item[(iv)] If $G$ is an
automorphism group of $\cS$ which is transitive on $\cP$, then
$\cent{\aut\cS}G$ is intransitive on $\cP$.
\item[(v)] If $G$ is an automorphism group of $\cS$
acting faithfully and flag transitively, then
$|G|\leq|G_x|^{12}$ for all $x\in\cP$.
\end{itemize}
\end{lemma}
\begin{proof}
Claim~(i) is shown in~\cite{bt}.
To show~(ii), let $g$ be an automorphism with order $2$ and assume that $g$ has
no fixed point. Then, by~(i), there is a point $x\in\cP$, such that $x\sim
xg$. Suppose that $L$ is the line that is incident with $x$ and $xg$. Then
the image $Lg$ 
of $L$ is incident with $xg$ and $xg^2=x$, and so $Lg=L$. Thus $L$
must be a fixed line of $g$. If $G$ is an automorphism group with even order 
then $G$ contains an automorphism with order $2$. 
If $G$ is semiregular on $\cP$ then $g$ has no fixed point in $\cP$. Thus, 
by the argument above, $g$ must have a fixed line, and so $G$ cannot be
semiregular on $\cL$. Thus (ii) is proved. In claim~(iii),
as $x\sim y_1$, the point $y_2=x{g_2}$ is collinear with
$y_1g_2=x{g_1g_2}$. Similarly,
$y_1=x{g_1}$  is collinear with $y_2g_1=x{g_2g_1}=x{g_1g_2}$.
Hence if $x\neq x{g_1g_2}$, then the gonality of $\cS$ would be
at most $4$, which is a contradiction.
Let us now show (iv).
Set $C=\cent{\aut\cS}G$ and assume that $C$ is transitive on $\cP$.
Let $x$ and $y$ be vertices of $\cS$ such that $x\sim y$. Then there is some
$g\in G$ such that $xg=y$. On the other hand, as $\cS$ is thick and its
gonality is at least $6$,
we can choose distinct
vertices $y_1$ and $y_2$ such that $x\sim y_1$, $x\sim y_2$, $y\not\sim y_1$,
and $y\not\sim y_2$. By assumption,
$C$ is transitive, and so there are
$c_1,\ c_2\in C$ such that $x{c_1}=y_1$ and $x{c_2}=y_2$. Then we obtain
that $y_1g=y_2g=x$, which is a contradiction, and so (iv) is valid.

Finally, we verify~(v). Suppose first
that $\cS$ is a generalized hexagon with order $(s,t)$, let
 $x\in\cP$ and let $G_x$ denote the
stabilizer in $G$ of $x$. Since $G$ is flag-transitive, $G_x$ must be
transitive on the $t+1$ lines that are incident with $x$ and,
in particular,
$|G_x|\geq t+1$.
Therefore, using the Orbit-Stabilizer Theorem and
the inequality $s\leq t^3$ (see~\cite{hr}
and \cite[Theorem~1.7.2(ii)]{vm}),
$$
\frac{|G|}{|G_x|}=|\cP|=(1+s)(1+st+s^2t^2)\leq (1+t^3)(1+t^4+t^8)\leq
(1+t)^{11}\leq |G_x|^{11},
$$
and the statement for hexagons follows.
If $\cS$ is a generalized octagon with order $(s,t)$, then,
using the inequality $s\leq t^2$
(see~\cite{hig} and~\cite[1.7.2(iii)]{vm}), we obtain similarly that
$$\frac{|G|}{|G_x|}=|\cP|=(1+s)(1+st)(1+s^2t^2)\leq (1+t^2)(1+t^3)(1+t^6)\leq
(1+t)^{11}\leq |G_x|^{11},$$
and the statement for octagons also follows.
 \end{proof}

We note that a generalized hexagon or octagon is a self-dual structure, and so
the dual of a true statement is also true. For instance,  taking
the dual of
statement~(iv), we obtain the following fact:
{\em if $G$ is a line-transitive automorphism group of
  $\cS$, then $\cent{\aut S}G$ is intransitive on the lines}. In this paper we
do not state the dual of a each of the results, but we often use the dual
statements in our
arguments.

We will also need the following group theoretic lemma.
Recall that a group $G$ is said to be {\em almost simple} if it has a unique minimal
normal subgroup $T$ which is non-abelian and simple. In this case, $T$ is the
{\em socle} of $G$ and the group $G$ can be
considered as a subgroup of the automorphism group of $T$ containing all inner
automorphisms.

\begin{lemma}\label{lemma2}
(a)
Let $S$ be an almost simple group with socle $T$
and let $H$ be a maximal subgroup of $S$ such that $T\not\leq H$.
Then $\norm{T}{H\cap T}=H\cap T$.

(b) Suppose that $T_1,\ldots,T_k$ are pairwise isomorphic finite simple groups and, for
$i=2,\ldots,k$, the map $\alpha_i:T_1\rightarrow T_i$ is an isomorphism. Then
the subgroup
$$
D=\{(t,\alpha_2(t),\ldots,\alpha_k(t))\ |\ t\in T_1\}
$$ is self-normalizing in $T_1\times\cdots\times T_k$.
\end{lemma}
\begin{proof}
(a)
If $S$ and $T$ are as in the lemma, then
$H\cap T\unlhd H$. Hence
$H\leq \norm S{H\cap T}$. Note that $S$ can be considered as a primitive group
acting on the right coset space $[S:H]$ with point-stabilizer $H$. Since the 
socle of such a primitive group is non-regular, we obtain that $H\cap T\neq
1$. Hence  
$H\cap T$ is a proper, non-trivial subgroup of $T$, which cannot be normal in $S$. Thus, since $H$ is a maximal
subgroup of $S$, we obtain that
$\norm S{H\cap T}=H$.
Hence $\norm T{H\cap T}=\norm S{H\cap T}\cap T=H\cap T$.

(b) Let $G=T_1\times\cdots\times T_k$ and let $(t_1,\ldots,t_k)\in\norm{G}D$. Then,
for all $t\in T_1$,
$$
(t,\alpha_2(t),\ldots,\alpha_k(t))^{(t_1,\ldots,t_k)}=\left(t^{t_1},\alpha_2(t)^{t_2}\ldots,
\alpha_k(t)^{t_k}\right)\in D.
$$
Thus, for all $i\in\{2,\ldots,k\}$, we obtain that $\alpha_i(t^{t_1})=\alpha_i(t)^{t_i}$. Therefore
$t_1\alpha_i^{-1}(t_i)^{-1}\in\cent{T_i}t$. As this is true for all $t\in T_1$, we obtain that
$t_1\alpha_i^{-1}(t_i)^{-1}\in Z(T_1)$. As $T_1$ is a non-abelian, finite, simple group, this yields that
$\alpha_i(t_1)=t_i$. Hence $(t_1,\ldots,t_k)\in D$, and so $\norm GD=D$.
\end{proof}

\section{Hexagons and Octagons with primitive automorphism group}

The structure of a finite primitive permutation group is described by the O'Nan-Scott Theorem
(see~\cite[Sections~4.4--4.5]{cam} or~\cite[Section~4.8]{dm}). In the mathematics literature, one can find
several versions of this theorem, and in this paper we use the version that can, for instance, be found
in~\cite[Section~3]{bad:quasi}. Thus we distinguish between 8 classes of finite primitive groups, namely {\sc
HA}, {\sc HS}, {\sc HC}, {\sc SD}, {\sc CD}, {\sc PA}, {\sc AS}, {\sc TW}. A
description of these classes can be found below.

Recall that in a finite group $G$, the {\em socle} of $G$ is the product of the
minimal normal subgroups in $G$ and it is denoted by $\soc G$. In fact, $\soc
G$ is the direct product of the minimal normal subgroups of $G$. As a minimal
normal subgroup of $G$ is a direct product of pairwise isomorphic finite
simple groups, the socle of $G$ is also the direct product of
finite simple groups.

Suppose that $G_1,\ldots,G_k$ are groups, set $G=G_1\times\cdots\times G_k$,
and, for
$i\in\{1,\ldots,k\}$, let $\varphi_i$ denote the natural projection map
$\varphi_i:G\rightarrow G_i$. A subgroup $H$ of $G$ is said to be {\em subdirect
with respect to the given direct decomposition of $G$} if
$\varphi_i(H)=G_i$ for $i=1,\ldots,k$. If the $G_i$ are non-abelian
finite simple groups then the $G_i$ are precisely the minimal normal
subgroups of $G$. In this case,
a subgroup $H$ is said to be {\em
  subdirect} if it is  subdirect with respect to the decomposition
of $G$ into the direct product of its minimal normal subgroups.
If $G$ is a finite group then the holomorph $\hol G$ is defined as the
semidirect product $G\rtimes \aut G$.

The O'Nan-Scott type of a finite primitive permutation group $G$ can be
recognized from the structure and the
permutation action of $\soc G$. Let $G\leq\sym\Omega$
be a finite primitive
permutation group, let $M$ be a minimal normal subgroup of $G$, and
let $\omega\in\Omega$. Note that $M$ must be transitive on $\Omega$.
Further, $M$ is a characteristically simple group, and so it is isomorphic to
the direct product of pairwise non-isomorphic finite simple groups.
The
main characteristics of $G$ and $M$ in each primitive type
are as follows.

\begin{enumerate}
\item[HA] $M$ is abelian and regular, $\cent GM=M$ and $G\leq \hol M$.
\item[HS] $M$ is non-abelian, simple, and regular;
$\soc G=M\times\cent GM\cong M\times M$ and
$G\leq\hol M$.
\item[HC] $M$ is non-abelian, non-simple, and regular; $\soc
G=M\times\cent GM\cong M\times M$ and $G\leq\hol M$.
\item[SD] $M$ is non-abelian and non-simple; $M_\omega$ is a
simple subdirect subgroup of $M$ and $\cent GM=1$.
\item[CD] $M$ is non-abelian and non-simple; $M_\omega$ is a
non-simple subdirect subgroup of $M$ and  $\cent GM=1$.
\item[PA] $M$ is non-abelian and non-simple; $M_\omega$ is a
not a subdirect subgroup of $M$ and $M_\omega\neq 1$; $\cent GM=1$.
\item[AS] $M$ is non-abelian and simple; $\cent GM=1$, and so $G$ is
an almost simple group.
\item[TW] $M$ is non-abelian and non-simple; $M_\omega=1$; $\cent GM=1$.
\end{enumerate}

We pay special attention to the groups of type AS. In this class, the group
$G$ has a unique minimal normal subgroup which is non-abelian and
simple. Therefore $G$ is isomorphic to a subgroup of $\aut T$ which contains
all inner automorphisms. Such an abstract group is referred to as {\em almost
  simple}.
The next result shows that under certain conditions a primitive automorphism
group of a generalized hexagon or octagon must be an almost simple group.

\begin{theo}\label{th2}
If $G$ is a point-primitive, line-primitive and flag-transitive
group of automorphisms of a generalized hexagon or octagon,
then the type of $G$ must be {\rm AS} on both the points and the lines.
In particular, $G$, as an abstract group, must be
almost simple.
\end{theo}

Theorem~\ref{th2} is a consequence of the following lemma.

\begin{lemma}\label{prlemma}
If $G$ is a
group of automorphisms of a generalized hexagon or octagon
$\cS=(\cP,\cL,\inc)$ then the following holds.
\begin{enumerate}
\item[(i)] If $G$ is primitive on $\cP$ then the type of $G$ on $\cP$ is
not {\rm HA}, {\rm HS}, {\rm HC}.
Dually, if $G$ is primitive on $\cL$ then the type of $G$ on $\cL$ is
not {\rm HA}, {\rm HS}, or {\rm HC}.
\item[(ii)] If $G$ is flag-transitive and it is primitive on $\cP$ then
the type of $G$ on $\cP$ is not  {\rm PA} or {\rm SD}. Dually,
if $G$ is flag-transitive and it is primitive on $\cL$ then
the type of $G$ on $\cL$ is not  {\rm PA} or {\rm SD}.
\item[(iii)] If $G$ is flag-transitive and it is
primitive on both $\cP$ and $\cL$, then the
O'Nan-Scott type of $G$ on $\cP$ and on $\cL$ is not {\rm SD} or {\rm TW}.
\end{enumerate}
\end{lemma}
\begin{proof}
Let $\cS$ and $G$ be as assumed in the theorem.
Suppose  further that $G$ is primitive on $\cP$
and let $M$ be a fixed minimal normal subgroup
of $G$. In this case,
$M=T_1\times\cdots\times T_k$ where the $T_i$ are finite
simple groups; let $T$ denote the common isomorphism type of the $T_i$.

(i)
As $M$ is transitive on $\cP$,
Lemma~\ref{lemma1}(iv) implies that $\cent GM$ must
be intransitive. Since $\cent GM$ is a normal subgroup of $G$, we obtain that
$\cent GM=1$.
Hence the O'Nan-Scott type of $G$ on $\cP$ is not HA, HS,
HC. The dual argument proves the dual statement.

(ii)
Assume now that $G$ is flag-transitive and it is primitive on $\cP$.
We claim that the O'Nan-Scott type of $G$ on $\cP$
is not PA or CD. Assume by contradiction that this O'Nan-Scott type is PA or
CD.
In this case $\cP$ can be identified with the Cartesian product $\Gamma^\ell$
in such a way that $G$ can be viewed as a subgroup of the wreath product
$H\wr\sy\ell$ where $H$ is a primitive subgroup of $\sym\Gamma$
and the projection of $G$ into $\sy\ell$ is transitive. Set $N=\soc
H$ and let $\gamma\in\Gamma$.  We must have that $N^\ell$, considered as a
subgroup of $H\wr\sy\ell$, is a subgroup of $G$, and, in fact, $N^\ell=\soc
G=M$.
Further, we have the following two possibilities.
\begin{itemize}
\item[PA] If the type of $G$ is PA then the type of $H$ is AS and we have that
$N\cong T$, $\ell=k$ and $N_\gamma$ is a proper subgroup of $N$.
\item[CD] If the type of $G$ is CD,
then the type of $H$ is SD, $N\cong T^s$ where $s\geq 2$ and $s=k/\ell$.
In this case,
$N_\gamma$ is a diagonal subgroup in $N$
which is isomorphic to $T$.
\end{itemize}

Since $H$ is primitive on $\Gamma$,
the normal subgroup $N$  must be
transitive on $\Gamma$. If $\gamma\in\Gamma$, then $H_\gamma$ is a
maximal subgroup of $H$. Thus Lemma~\ref{lemma2} implies that $\norm
N{N_\gamma}=N_\gamma$ (part~(a) of the lemma applies in the PA case, and 
part~(b) applies in the CD case).
Suppose that $\gamma,\ \delta\in\Gamma$ such that
$N_\gamma=N_\delta$. Then there is $n\in N$ such that $\gamma n=\delta$, and
$(N_\gamma)^n=N_\delta=N_\gamma$. Hence $n$ normalizes $N_\gamma$, and so
$n\in N_\gamma$, and we obtain that $\gamma=\delta$. Therefore different points of
$\Gamma$ must have different stabilizers in $N$.

Let $\alpha$ be an arbitrary element of $\Gamma$ and consider the
point $x\in\cP$ represented by the $\ell$-tuple
$(\alpha,\alpha,\ldots,\alpha)$. We claim first that there exists
a point $y\sim x$ such that every entry of the $\ell$-tuple
representing $y$ is equal to $\alpha$, except for one entry.
Indeed, let $y$ be any point collinear with $x$. Then, if the
claim were not true, we may assume without loss of generality that
$y$ is represented by $(\beta_1,\beta_2,\ldots)$, where
$\beta_1\neq\alpha\neq\beta_2$. By the argument in the previous paragraph,
the stabilizers in $N$ of $\alpha$ and $\beta_1$ are distinct, and so there exists $g\in N_\alpha$ such that
$\beta_1':=\beta_1g\neq\beta_1$. Put
$\overline{g}=(g,1_N,1_N,\ldots,1_N)$ ($\ell$ factors) and
$y'=y{\overline{g}}$. Let $h\in N_{\beta_2}$ be such that
$\alpha':=\alpha h\neq\alpha$ (such an $h$ exists by
the argument in the previous paragraph). Put
$\overline{h}=(1_N,h,1_N,1_N,\ldots,1_N)$ ($\ell$ factors), and put
$x'=x{\overline{h}}$. Then $x'\neq x$, and both $y$ and $y'$ are
fixed under $\overline{h}$. Since $y\sim x\sim y'$, we deduce
$y\sim x'\sim y'$. This implies (because the gonality of $\cS$ is
at least $5$) that $x,x',y,y'$ are incident with a common line.
But all entries, except the second, of $x'$ are equal to $\alpha$.
Thus our claim is proved.

So we may pick $y\sim x$ with
$y=(\beta,\alpha,\alpha,\ldots,\alpha)$ ($\ell$ entries) and $\beta\neq
\alpha$. By the flag-transitivity, there exists $g\in G_x$ mapping
$y$ to a point not collinear with $y$. There are two
possibilities.

\begin{enumerate}
\item[(a)] We can choose $g$ such that the first entry of $yg$ is
equal to $\alpha$.
\item[(b)] For every such $g$, the first entry of $yg$ differs from
$\alpha$.
\end{enumerate}

In case~(a), as $x=(\alpha,\ldots,\alpha)$ and $g\in G_x$,
we may suppose without loss of generality
that $y':=yg=(\alpha,\beta',\alpha,\ldots,\alpha)$. Choose
$h,h'\in N$ such that $\alpha h=\beta$, and $\alpha {h'}=\beta'$.
Put $\overline{h}=(h,1_N,\ldots,1_N)$ and
$\overline{h}'=(1_N,h',1_N,\ldots,1_N)$. Then $\overline{h}$ and
$\overline{h}'$ commute and Lemma~\ref{lemma1}(iii) implies that
$x=x{\overline{h}\overline{h}'}$. Hence $\alpha=\beta=\beta'$, a
contradiction.

In case~(b), we consider an arbitrary such $g$ and put $z=yg$. Also,
consider an arbitrary $g'\in G_x$ not preserving the first
component of $\Gamma\times\Gamma\times\cdots\times\Gamma$. By
assumption, $y{g'}$ is incident with the line through $x$ and
$y$, and we put $z'=y{g'}$. If we now let $y$ and $y'$ in the
previous paragraph play the role of $z$ and $z'$, respectively, of
the present paragraph, then we obtain a contradiction again.

Thus we conclude that the type of $G$ on $\cP$ is not PA or CD
and the dual statement can be verified using the dual argument.

(iii) Suppose that $\cS$ is a generalized hexagon or octagon and $G$ is a
group of automorphisms such that $G$ is flag-transitive and $G$ is primitive
on $\cP$ and $\cL$ of type either SD or TW.
First we claim that
$\cS$ must be a generalized hexagon and
$\gcd(s,t)=1$. If $\cS$ is a generalized octagon with order $(s,t)$,
then either $|\cP|$ or $|\cL|$ must be odd. However,
the degree of a primitive group with type SD or
TW is the size of a minimal normal subgroup,
which is even, as it is a power of the size of a
non-abelian finite simple group.
Therefore $\cS$ must be a hexagon as claimed.
Assume now by contradiction that $\gcd(s,t)\neq 1$ and
consider the subgroup $T_1$ of the socle $M$. Since $G$ is either SD or TW on
$\cP$ and also on $\cL$ we have that $T_1$ is semiregular on both $\cP$ and on
$\cL$. However, as $T_1$ is a non-abelian finite simple group, $T_1$ has even
order, and this is a contradiction, by Lemma~\ref{lemma1}(ii).

So we may suppose for the remainder of this proof that $\cS$ is a generalized hexagon with parameters $(s,t)$
such that $\gcd(s,t)=1$. Note that the number of lines is $(t+1)(1+st+s^2t^2)$, and the number of points is
$(s+1)(1+st+s^2t^2)$. If $G$ has the same O'Nan-Scott type on the set of points and the set of lines, then
$|\cP|=|\cL|$, which implies $s=t$. Since $\gcd(s,t)=1$, this is impossible, and we may assume without loss of
generality that the type of $G$ is SD on $\cP$ and it is TW on $\cL$. Hence
$|\cP|=(s+1)(1+st+s^2t^2)=|T|^{k-1}$ and $|\cL|=(t+1)(1+st+s^2t^2)=|T|^{k}$. 
Thus $|T|=(t+1)/(s+1)$ and so
$t=s|T|+|T|-1$.

We digress in this paragraph to show that 
the order of the non-abelian finite simple group $T$ is
divisible by~4. It seems to be well-known that 
this assertion follows immediately from the
Feit-Thompson Theorem which states that $|T|$ is even. The following simple
argument was showed to us by Michael Giudici in private communication. 
Recall that the right-regular
representation $\varrho$ of $T$ is a homomorphism from $T$ to $\sym T$ that
maps $t\in T$ to the permutation $\varrho(t)\in\sym T$ where $\varrho(t)$ is
defined by the equation
$x\varrho(t)=xt$ for all $x\in T$. It is easy to see that $\varrho(T)$ is a
regular subgroup of $\sym T$; that is $\varrho(T)$ is transitive, and, for 
all $t\in T\setminus\{1\}$,
$\varrho(t)$ has no fixed-points.
Now $\varrho(T)\cong T$ and 
$\varrho(T)\cap \Alt T$ is a normal
subgroup of $\varrho(T)$ with index at most 2. Thus $\varrho(T)\leq \Alt T$,
and so every element of $\varrho(T)$ is an even permutation on $T$. 
By the
Feit-Thompson Theorem referred to above, we have that $T$ contains an
involution $g$. Since $\varrho(g)$ is also an involution, it must be the
product of disjoint transpositions. As 
$\varrho(g)$ is an even permutation, the number of
transpositions in $\varrho(g)$ must be even. Further, as $\varrho(g)$ has no fixed-points,
every element of $T$ must be involved in precisely one of  these
transpositions. This implies that $4\mid|T|$, as claimed.

We now continue with the main thrust of the proof. In order to derive a
contradiction, we show that the equations for  $s$, $t$ and $|T|$ above imply
that $4\nmid |T|$. Indeed, note that $st$ is a square, and so, as
$\gcd(s,t)=1$, we have that $t$ must be a
square. If $4$ divides $|T|$ then
$t=s|T|+|T|-1\equiv 3\pmod 4$.
However, $3$ is not a square modulo $4$, which gives the desired
contradiction. Hence, in this case, $G$ cannot be primitive with type SD or TW.
\end{proof}

The reader may wonder whether it is possible for an abstract group 
$G$ to have two faithful primitive permutation actions, one with type TW and
one with type SD. Gross and Kov\'acs in~\cite{gk} show that if $G$ is a
twisted wreath product of $\alt 5$ and $\alt 6$ where the twisting subgroup in
$\alt 6$ is
isomorphic to $\alt 5$, then $G$ is isomorphic to the straight wreath
product $\alt 5\wr \alt 6$. Hence in this case $G$ can be a primitive 
permutation group of type TW and also of type SD.

Now we can prove Theorem~\ref{th2}.

\medskip

\noindent{\sc Proof of Theorem~\ref{th2}.}
Suppose that $G$ is a point-primitive, line-primitive and flag-transitive
group of automorphisms of a generalized hexagon or octagon.
Using parts~(i)--(iii) of Lemma~\ref{prlemma}, we obtain that the type of
$G$ on both the points and lines
must be AS. In particular $G$, as an abstract group,  must be almost simple.
\hfill$\Box$

\section{Hexagons and octagons with an almost simple automorphism group}

In this section we prove the following theorem.

\begin{theo}\label{mainas}
If $\cS$ is a generalized hexagon or octagon and
$G$ is a flag-transitive and point-primitive automorphism group of $\cS$,
then $G$ is not isomorphic to an alternating or symmetric group with degree
at least $5$.
\end{theo}

Our strategy to prove Theorem~\ref{mainas} is to show that a maximal subgroup
of an alternating or symmetric group cannot be a point-stabilizer.
To carry out this strategy, we need some arithmetic results about the maximal
subgroups of $\alt n$ and $\sy n$.

\begin{lemma}\label{107}
If $n\in\N$ and $n\geq 107$ then
\begin{equation}\label{est1}
n^{12+12\lfloor\log_2 n\rfloor}\leq n!/2.
\end{equation}
\end{lemma}
\begin{proof}
Checking the numbers between~107 and~208,
we can see that~\eqref{est1} holds for all $n\in\{107,\ldots,208\}$.
So suppose without loss of generality in the remaining of this proof
that $n$ is at least  $209$.
The Stirling Formula gives, for each $n\geq 1$,
that there is $\vartheta_n\in[0,1]$ such that
$n!=(n/e)^n\sqrt{2\pi n}e^{\vartheta_n/(12n)}$ (see~\cite[Theorem~2, Chapter~XII]{lang}),
which gives that $(n/e)^n\leq n!/2$.
We claim that $n^{n/2}\leq (n/e)^n$ for $n\geq 8$. Easy calculation shows that
the inequality holds for $n=8$. We
assume that it holds for some $n$ and prove by induction that it holds for
some $n+1$. Let us compute that
$$
\left(\frac{(n+1)^{(n+1)/2}}{n^{n/2}}\right)^2=\frac{(n+1)^{n+1}}{n^n}=
(n+1)\left(\frac{n+1}{n}\right)^n
$$
and
$$
\left(\frac{\left((n+1)/e\right)^{n+1}}{(n/e)^{n}}\right)^2=e^{-2}\left(\frac{(n+1)^{n+1}}{n^n}\right)^2=
e^{-2}(n+1)^2\left(\frac{n+1}{n}\right)^{2n}.
$$
This shows that
$$
\frac{(n+1)^{(n+1)/2}}{n^{n/2}}\leq
\frac{\left((n+1)/e\right)^{n+1}}{(n/e)^{n}},
$$
and the assumption that $n^{n/2}\leq (n/e)^n$ gives
the claimed inequality for $n+1$. Therefore it suffices to show that
$n^{12+12\lfloor\log_2 n\rfloor}\leq n^{n/2}$, and, in turn, we only have to show
  that $12+12\log_2 n\leq n/2$ for $n\geq 209$. Again, easy computation shows
  that the inequality holds for $n=209$. Since $x\mapsto 12+12\log_2 x$ is a
  concave function and $x\mapsto x/2$ is a linear function, the inequality
  must hold for all $n\geq 209$.
\end{proof}

\begin{lemma}\label{smallprim}
Suppose that $G$ is an alternating or symmetric group with degree $n$ ($n\geq
5$) and
$H$ is a primitive and maximal subgroup of $G$ such that $|H|^{12}\geq |G|$.
Then $G$ and $H$ must be as one of the groups
in the table of Appendix~\ref{app1}.
\end{lemma}
\begin{proof}
Suppose that $H$ is a primitive and maximal subgroup of $G$.
Using the classification of maximal subgroups of the alternating and symmetric
groups~\cite{max} and Mar\'oti's Theorem~\cite[Theorem~1.1]{maroti}, we have that one of the following
must hold:
\begin{enumerate}
\item[(1)] $n=k^\ell$ for some $k\geq 5$ and $\ell\geq 2$ and
$H$ is permutationally
isomorphic to $(\sy k\wr\sy\ell)\cap G$ in product action;
\item[(2)] $G$ is isomorphic to $\mat n$ for $n\in\{11,12,23,24\}$
in its $4$-transitive action;
\item[(3)] $|G|< n^{1+\lfloor\log_2 n\rfloor}$.
\end{enumerate}

Suppose that case~(1) is valid and let
$H$ be permutationally isomorphic to the group
$(\sy k\wr \sy\ell)\cap G$ in product action for some $k\geq 5$ and $\ell\geq 2$.
Then we obtain that
$$
|H|^{12}\leq (k!)^{12\ell}\cdot(\ell!)^{12}.
$$
We claim that $(k!)^{12\ell}\cdot(\ell!)^{12}< (k^\ell)!/2$ except for finitely
many pairs $(k,\ell)$. First note that all primes $p$ dividing
$(k!)^{12\ell}\cdot(\ell!)^{12}$ will also divide $(k^\ell)!/2$.
For an integer $x$, let $|x|_p$ denote the largest non-negative
integer $\alpha$ such that $p^\alpha|x$.
It suffices
to show that, there are only finitely many pairs $(k,\ell)$
such that
$|(k!)^{12\ell}\cdot(\ell!)^{12}|_{p}\geq
|(k^\ell)!/2|_p$, where $p$ is an arbitrary prime which is not greater than
$\max\{k,\ell\}$.
It is routine to check that if $x$ is an integer then
\begin{equation}\label{eq1}
|x!|_p=\sum_{u=1}^\infty
\left\lfloor\frac{x}{p^u}\right\rfloor\leq\sum_{u=1}^\infty
\frac{x}{p^u}=\frac{x}{p}\sum_{u=0}^\infty\frac{1}{p^u}=\frac{x}{p}\cdot\frac{p}{p-1}
=\frac{x}{p-1}.
\end{equation}
Thus
$$
|(k!)^{12\ell}\cdot(\ell!)^{12}|_{p}\leq
12\ell\frac{k}{p-1}+12\frac{\ell}{p-1}=\frac{12\ell k+12\ell}{p-1}\leq
\frac{24\ell k+24\ell}{p}
$$
Clearly, $k^\ell\geq 8$. Further, as $k\geq 5$, $\ell\geq 2$, and $p\leq \max\{k,\ell\}$, we
obtain that $p^2\leq k^\ell$. Hence we obtain from the first equality
in~\eqref{eq1} that
$$
|(k^\ell)!/2|_p\geq \frac{k^\ell}{p}.
$$
Routine computation shows  that the set of pairs $(k,\ell)$
for which $k\geq 5$ and $\ell\geq 2$ and
$24\ell k+24\ell\geq k^\ell$ is
$\{5,\ldots,48\}\times\{2\}
\cup \{5,\ldots,8\}\times\{3\}$.
Then checking
finitely many possibilities it is easy to compute that
$(k!)^{12\ell}\cdot(\ell!)^{12}\geq (k^\ell)!/2$
if and
only if $(k,l)\in\{5,\ldots,10\}\times\{2\}$. In particular, the degree of $H$
is at most~100.

(2) Easy computation shows that $|\mat n|^{12}\geq n!/2$ for
$n\in\{11,12,23,24\}$.

(3) Lemma~\ref{107} shows that
if $n\geq 107$ then $n^{12+12\lfloor\log_2n\rfloor}\leq
n!/2$. Hence if $n\geq 107$ and $H$ is a maximal subgroup  of $\alt n$ or
$\sy n$ which is as in part~(3) of the theorem, then
$|H|^{12}< n!/2$.
Thus, in this case,
the degree of $H$ must be at most 106.

Summarizing the argument above: 
if $H$ is a primitive maximal subgroup of $G$ such that $|H|^{12}\geq |G|$ then
the degree of $H$ is at most~106. It remains to prove that $H$ must be one of the groups in the table in
Appendix~\ref{app1}. Various classifications of primitive groups of small degree can be found in the
literature; for convenience we use the classification
 by Roney-Dougal~\cite{rd}, as it
can be accessed through the computational algebra system {\sf GAP}~\cite{gap}.
In what follows we explain how we obtained the table in Appendix~\ref{app1}
using the {\sf GAP} system.
First, for a fixed $n\in\{5,\ldots,106\}$, 
let $P_n$ denote the list of of primitive groups
with degree $n$. For  $H\in P_n$ we check whether or not
$H\leq \alt n$. Then we check whether $|H|^{12}\geq |G|$ where $G$ is either
$\alt n$ (if $H\leq\alt n$) or $\sy n$ (otherwise).  If $H$ satisfies this
condition then we keep it in $P_n$, otherwise we erase it from $P_n$. 
The next step is to
eliminate those groups which are clearly not maximal subgroups in $\alt n$ or
$\sy n$. If $H_1,\ H_2\in P_n$ such that $H_1,\ H_2\leq \alt n$ and  $H_1<H_2$
then $H_1$ is erased from $P_n$.
Similarly, if $H_1,\ H_2\not \leq \alt n$ such that $H_1<H_2$,
then $H_1$ is thrown away. We do this calculation for all
$n\in\{5,\ldots,106\}$ and the subgroups $H$ that we obtain are in
Appendix~\ref{app1}.
\end{proof}

Let us note that Lemma~\ref{smallprim} is not an ``if and only if''
statement. Indeed, the table in the appendix may be
redundant in the sense, that a subgroup in the table may not be
maximal in $\alt n$ or $\sy n$. 

Let us now prove Theorem~\ref{mainas}.

\bigskip

\noindent{\sc Proof of Theorem~\ref{mainas}}.
Suppose that $\cS=(\cP,\cL,\inc)$ is a generalized hexagon or octagon and $G$
is a point-primitive, flag-transitive automorphism group of $\cS$
such that $G$ is
isomorphic to $\alt n$ or $\sy n$ with some $n\geq 5$.
By Buekenhout and Van
Maldeghem \cite{Bue-Mal}, we may assume that $n\geq 14$.
Let $x\in\cP$. Then $G_x$, as a subgroup of $\sy n$,
is either intransitive, or it is
transitive and imprimitive, or it is primitive. We consider these three cases below.

{\em $G_x$ is intransitive.} Here, $G_x$ is the stabilizer 
in $G$ of a partition of the underlying set into two
blocks, one with size $k$ and one with size $\ell$, where $k+\ell=n$,
$k\neq\ell$. Let us also allow here the case when $k=\ell$, though in this
case $G_x$ may not be intransitive. Assume without loss of
generality that $k\leq \ell$.
Then $G_x$ contains a
subgroup 
isomorphic to $(\alt k\times\alt \ell)\rtimes C_2$. Hence the
points of $\cS$ can be labelled with the subsets of $\{1,2,\ldots,n\}$ of size $k$. We may label $x$ as
$\{1,2,\ldots,k\}$. Let $k_1<k$ be maximal with the property that there is a point $y$ of $\cS$ collinear with
$x$ and the label of $y$ intersects $\{1,2,\ldots,k\}$ in $k_1$ elements. Without loss of generality, we may
assume that $y\sim x$ has label $\{1,2,\ldots,k_1,k+1,\ldots,2k-k_1\}$. First suppose that $k_1=k-1$. Note
that, since the permutation rank of $G$ is at least $4$, we may assume $k\geq 3$. By transitivity of $G_x$ on
$\{1,2,\ldots,k\}$, and by transitivity of the pointwise stabilizer of $\{1,2,\ldots,k\}$ on the complement
$\{k+1,k+2,\ldots,n\}$, every point with a label sharing exactly $k-1$ elements with $\{1,2,\ldots,k\}$ is
adjacent with $x$. An arbitrary element $g$ of $G_x$ now maps $y$ onto a point $y'$ with label, without loss of
generality, either $\{1,2,\ldots,k-1,k+2\}$ or $\{2,3,\ldots,k,k+1\}$ or $\{2,3,\ldots,k,k+2\}$. In the first
two cases $y'$ is collinear with $y$. Since, by flag-transitivity, we can choose $g$ such that it does not
preserve the line $xy$, and hence does not map the point $y$ onto a collinear point, we may assume that the
point $y'$ with label $\{2,3,\ldots,k,k+2\}$ is not collinear with $y$, and hence has distance $4$ to $y$ (in
the incidence graph). But now the automorphism $(1\;k+1)(k\;k+2)$ (if $k'=k+1$) fixes both $y$ and $y'$, but
not $x=y\JJoin y''$. Hence $k_1<k-1$.

Now the automorphism $(k-1\;k)(k+1\;2k-k_1+1)$ belongs to $G_x$ and maps $y$ to a point $z$ whose label shares
$k-1$ elements with $y$. Hence $z$ cannot be collinear with $y$ (otherwise, mapping $y$ to $x$, the image of
$z$ produces a point with a label contradicting the maximality of $k_1$ which is less than $k-1$). On the other
hand, $z$ is collinear with $x$. If $k_1>0$, then the automorphism $(1\; k+2)(k-1\;k)$ belongs to $G$,
preserves $y$ and $z$, but not $x=y\JJoin z$. Now suppose that $k_1=0$. If $2k+1<n$, then the automorphism
$(1\;2k+2)(2\;3)$ fixes $y$ and $z$, but not $x=y\JJoin z$, a contradiction. If $2k+1=n$, then, by the
maximality of $k_1$, and the transitivity of $\alt k$, we see that there are precisely $k+1$ points collinear
with $x$ on which $G_x$ acts $2$-transitively. This easily implies that either $s=1$ or $t=0$, either way a
contradiction.

{\em $G_x$ is imprimitive.} Here $G_x$ is the stabilizer of a partition of the underlying set into $\ell$
blocks each with size $k$. Let $x$ be a point of $\cS$, which we may assume without loss of generality to
correspond to the partition of $\{1,2,\ldots,n\}$ into $\ell$ subsets of size
$k$ given by
$\{ik+1,ik+2,\ldots,ik+k\}$, $0\leq i<\ell$. We may assume that $\ell>2$, the case $\ell=2$ being completely
similar to the intransitive case, as noticed above.
(If $\ell=2$ then, as the number of point is greater that $4$, we may also assume that $k\geq 3$). 

We first claim that there is some point $y\sim x$ such that $y$
corresponds to a partition sharing at least one partition class
with $x$ (we will identify the points with their corresponding
partition). Let $y$ be any point collinear with $x$ and suppose
that $y$ has no partition class in common with $x$. If $k=2$, then
$\ell>6$ and so the automorphism $(1\;2)(3\;4)$ destroys at most 4
classes of $y$, while it fixes $x$. Hence the image $z$ of $y$ has
at least three classes $\{i_1,i_2\},\{i_3,i_4\},\{i_5,i_6\}$ in
common with $y$, and therefore we may assume that $y\not\sim z$.
The group generated by $(i_1\;i_2)(i_3\;i_4)$,
$(i_1\:i_3)(i_2\;i_4)$ and $(i_1\;i_5)(i_2\;i_6)$ fixes both $y$
and $z$ but cannot fix $x$, a contradiction. Suppose now $k>2$.
Then the automorphism $(1\;2\;3)$ destroys at most $3$ classes of
$y$ and maps $y$ to a point $z$ sharing at least $\ell-3$ classes
with $y$. This is at least one if $\ell>3$. If $\ell=3$, then
$k>3$ and hence some class of $y$ shares at least two elements
with some class of $x$. Without loss of generality, we may assume
that $1,2$ are in some class of $y$ and hence the automorphism
$(1\;2\;3)$ destroys at most two classes of $y$, resulting in the
fact that $z$ shares at least one class with $y$ again. Let this
common class be given by $\{i_1,i_2,i_3,\ldots\}$, where we may
suppose without loss of generality that $i_1,i_2,i_3$ do not
belong to a common class of $x$. The automorphism
$(i_1\;i_2\;i_3)$ fixes both $y$ and $z$, but not $x=y\JJoin z$, a
contradiction. Our claim is proved.

Now let $\ell_1$ be maximal with respect to the property that there exist two collinear points sharing $\ell_1$
classes. By the foregoing, $\ell_1>0$, and we may assume that the class $\{1,2,\ldots,k\}$ belongs to the point
$y\sim x$. Suppose that $\ell_1<\ell-2$. In particular, it follows from our assumptions that $\ell\geq 4$. It
also follows from our assumptions that there is a transposition $(j_1\;j_2)$ fixing $x$ and not fixing $y$.
Hence the automorphism $(1\;2)(j_1\;j_2)$ preserves $x$ and maps $y$ to a point $z$ sharing $\ell-2$ classes
with $y$. By the maximality of $\ell_1$, we see that $y\not\sim z$. Also, $y$ and $z$ contain a common class
which is not a class of $x$. So there exist elements $j_3,j_4$ contained in a common class of both $y$ and $z$,
but belonging to different classes of $x$. The automorphism $(1\;2)(j_3\;j_4)$ fixes $y$ and $z$, but not
$x=y\JJoin z$, a contradiction. We have shown that $\ell_1=\ell-2$, and this now holds for all $\ell\geq 3$.

Now let $k_1$ be the maximal number of elements in the intersection of two distinct classes of two collinear
elements sharing $\ell-2$ classes. Note that $k_1\geq k/2>0$. First we show that $k_1<k-1$. So we assume by way
of contradiction that $k_1=k-1$. By transitivity of $G_x$, every point with a partition sharing $\ell-2$
classes with $x$ and for which the distinct classes share $k-1$ elements, is collinear with $x$. By
flag-transitivity and thickness, at least two such points $y',y''$ are not collinear with $y$. If the different
classes of $y'$ (compared with the classes of $x$) are the same as those of $y$, then, for $k>2$, the same
arguments as in the intransitive case lead to a contradiction. For $k=2$, $y''$ does not have this property
(since there are only three points with $\ell-2$ given partition classes), and we switch the roles of $y'$ and
$y''$ in this case. So $y'$ differs from $y$ in three or four classes. We distinguish between two
possibilities.

(1) \emph{$y$ and $y'$ differ in exactly three partition classes.} We may assume that $y$ contains the classes
$\{1,2,\ldots,k-1,k+1\}$ and $\{k,k+2,k+3,\ldots,2k\}$ (and the other classes coincide with classes of $x$).
Without loss of generality, there are two possibilities for $y'$. Either $y'$ contains the classes
$\{1,2,\ldots,k-1,2k+1\}$ and $\{k,2k+2,2k+3,\ldots,3k\}$ (and the other classes coincide with classes of $x$),
or $y'$ contains the classes $\{1,2,\ldots,k-2,k,2k+1\}$ and $\{k-1,2k+2,2k+3,\ldots,3k\}$ (and the other
classes coincide with classes of $x$). In the first case the automorphism $g=(k+1\; 2k+1\; k)$ maps $y$ onto
$y'$, and $y'$ onto a point collinear with $x$. Since $y'g$ is not collinear with $yg=y'$, we see that $g$ must
preserve $y\JJoin y'=yg\JJoin y'g=x$. But it clearly does not, a contradiction. In the second case the
automorphism $(k-1\; 2k+1) (k\; k+1)$ interchanges $y$ with $y'$, but does not fix $x=y\JJoin y'$, a
contradiction.

(2) Hence \emph{$y$ and $y'$ differ in four partition classes.} We take $y$ as in (1), and we can assume that
$y'$ contains the classes $\{2k+1,2k+2,\ldots,3k-1,3k+1\}$ and $\{3k,3k+2,3k+3,\ldots,4k\}$. Now the
automorphism $(k\; k+1)\; (3k\; 3k+1)$ interchanges $y$ with $y'$ without fixing $x=y\JJoin y'$, a
contradiction.

Hence we have shown $k_1<k-1$. But now the rest of the proof is similar to the last paragraph of the
intransitive case, where the subcase $k_1=0$ cannot occur. We conclude that $G_x$ is primitive on
$\{1,2,\ldots,n\}$.

{\em $G_x$ is primitive.}
By Lemma~\ref{lemma1}(v),
$|G|\leq |G_x|^{12}$, and so Lemma~\ref{smallprim} implies that $G$ and $G_x$
must be in the table of Appendix~\ref{app1}. Set $u=|\cP|=|G:G_x|$
and let $a(u)$ and $b(u)$ be the quantities defined before
Lemma~\ref{21}. Then Lemma~\ref{21} implies that if $\cS$ is a hexagon then
$a(u)^3\leq u$ and, if $\cS$ is an octagon, then $b(u)^2\leq u$.
For each
pair $(G,G_x)$ in Appendix~\ref{app1},
one can compute using, for instance, the {\sf GAP} computational algebra
system, the quantities $u$, $a(u)$, and $b(u)$. The computation shows
that
$a(u)^3> u$ and $b(u)^2> u$ holds in each of the cases.
The computation of $a(u)$ and $b(u)$ are presented in Appendices~\ref{app2}
and~\ref{app3}.
Therefore
none of the groups in Appendix~\ref{app1} can occur, and so we exclude this
case as well.

Thus $G$ cannot be an alternating or symmetric group.
\hfill$\Box$

Now we can prove our main theorem.

\noindent{\sc Proof of Theorem~\ref{simple}}. Suppose that $\cS$ and $G$ are as in the theorem. Then
Theorem~\ref{th2} implies that $G$ must be an almost simple group. Let $T$ denote the unique minimal normal
subgroup of $G$. Note that $T$ is a non-abelian simple group. By~\cite{Bue-Mal}, $T$ cannot be a sporadic
simple group, and by Theorem~\ref{mainas}, $T$ cannot be an alternating group. Thus $T$ must be a simple group
of Lie type and $G$ must be an almost simple group of Lie type. \hfill$\Box$

\section{Directions of future work}

Now that Theorem~\ref{simple} is proved, the next step in the full classification
of generalized hexagons and octagons satisfying the conditions of
Theorem~\ref{simple} is to treat the class of almost simple groups of Lie type. 
It is not our intention to be as detailed as possible 
regarding these groups, as we think the only worthwhile job now is to 
complete the classification in full.
We noted in the proof Lemma~\ref{prlemma} that in a generalized octagon either
the number of points or the number of lines is odd. Therefore it is meaningful
to investigate which almost simple groups of Lie type with odd
degree can occur in Theorem~\ref{simple}. Another possible task is to
use Lemma~\ref{lemma1} to characterize the case when the parameters are not
co-prime. 
We conclude this paper by presenting a couple of examples to 
illustrate that Lemmas~\ref{21} and~\ref{lemma1}  can be used, to
some extent, in this direction.
However, our examples also
show that a complete treatment of these groups is beyond the scope of this
paper and will probably require new ideas.

Let us assume that $G$ is an almost simple group of Lie type with socle $T$
and that $G$ is a group of automorphisms of a generalized hexagon or
octagon $\cS=(\cP,\cL)$ acting primitively both on the point set and on the line set,  
and transitively on the set of flags. 
Suppose, in addition, that the number $|\cP|$ of points is odd and let $x$ be a
point.  The possibilities
for $T$ and the point stabilizer $T_x$ can be found in~\cite{Kan:83,ls}. One
possibility, for instance, is that $q=3^{2m+1}$ with some $m\geq 1$, 
$T\cong {}^2{\sf G}_2(q)$ and
$|T_x|=q(q^2-1)$. We claim that it follows from our results that this case
cannot occur. Note that $|\cP|=q^2(q^2-q+1)$.
If $\cS$ is a hexagon, then Lemma~\ref{21} implies that 
$a(q^2(q^2-q+1))^3\leq q^2(q^2-q+1)$ (the function $a$ is defined before 
Lemma~\ref{21}). 
However, 
$a(q^2(q^2-q+1))^3\geq 3^{12m+3}$ which would imply that
$3^{12m+3}\leq 3^{8m+4}$ which does not hold for $m\geq 1$. Thus such a
hexagon does not exist, and similar argument shows that neither does such an
octagon. 

Another case is that $T\cong {\sf F}_4(q)$, 
$|T_x|=q^{16}(q^2-1)(q^4-1)(q^6-1)(q^8-1)$, and so $|\cP|=q^8(q^8+q^4+1)$. 
 Computer
calculation shows that among the prime-powers that are smaller then $10^4$,
there are 626 values of $q$ such that $a(|\cP|)^3\leq |\cP|$, and there are
625 such values with $b(|\cP|)^2\leq |\cP|$. Therefore Lemma~\ref{21} cannot
directly be used to exclude this case.

We conclude this paper with an  example that shows how 
Lemma~\ref{lemma1} may be applied. Let $\cS=(\cP,\cL)$ be as above and let us
assume that the parameters $s$ and $t$ of $\cS$ are not co-prime. By
Lemma~\ref{lemma1}(ii), an involution in $G$ 
either fixes a point or fixes a line.  Now if $G$ is isomorphic to $\alt n$ or
$\sy n$ with some $n\geq 5$, then, by possibly taking the
dual polygon, we may assume that a double transposition (in the natural representation of
$G$) is contained in a point stabilizer $G_x$. Therefore, as a subgroup of
$\sy n$, $G_x$ has minimal
degree at most 4 (see~\cite[page 76]{dm} for the definition of the minimal degree). Now if $G_x$ is primitive then~\cite[Example~3.3.1]{dm} shows that $n\leq
8$, and hence $G$ is ruled out  by~\cite{Bue-Mal}. This argument shows that under the
additional condition that $\gcd(s,t)\neq 1$, the proof of Theorem~\ref{mainas}
can be significantly simplified.

\section{Acknowledgments}

The first author was supported by the Hungarian Scientific Research Fund (OTKA) grant F049040; he is grateful
to the Ghent University for its hospitality while working on this paper. The second author is partly
supported by a research grant of the Fund for Scientific Research (FWO -- Vlaanderen); he is grateful to the
Computer and Automation Research Institute in Budapest for the hospitality while working on this paper.

We would like to thank the participants of the discussion following the
second author's lecture in August 2007 at the ``Permutation Groups''
workshop in Oberwolfach for their useful observations. In particular we
thank Pham Huu Tiep for suggesting a shortcut in the proof of Lemma~
\ref{prlemma} and Bill Kantor for his comments and his careful reading
of the paper.

\begin{appendices}
\include{table}
\end{appendices}
\end{document}

%% file: table.tex
\section{Table 1}\label{app1}

The following table contains the primitive maximal subgroups $H$ 
of $\alt n$ and $\sy n$ ($n\geq 5$) such that $|H|^{12}\leq n!$ and 
$|H|^{12}\leq n!/2$ if $H\leq \alt n$. Note that the table may contain
non-maximal subgroups; see the remarks at the end of the proof of
Lemma~\ref{smallprim}. The table was automatically generated from a {\sf GAP}
output, and so the notation follows the {\sf GAP} system.

\bigskip

{\tiny
$$
\begin{array}{|c|c|c|c|}
\hline\verb?C(5)?
\leq{\sf A}_{5}
&
\verb?D(2*5)?
\leq{\sf A}_{5}
&
\verb?AGL(1, 5)?
\leq{\sf S}_{5}
&
\verb?PSL(2,5)?
\leq{\sf A}_{6}
\\
\verb?PGL(2,5)?
\leq{\sf S}_{6}
&
\verb?7:3?
\leq{\sf A}_{7}
&
\verb?L(3, 2)?
\leq{\sf A}_{7}
&
\verb?AGL(1, 7)?
\leq{\sf S}_{7}
\\
\verb?ASL(3, 2)?
\leq{\sf A}_{8}
&
\verb?PSL(2, 7)?
\leq{\sf A}_{8}
&
\verb?PGL(2, 7)?
\leq{\sf S}_{8}
&
\verb?3^2:(2'A(4))?
\leq{\sf A}_{9}
\\
\verb?PGammaL(2, 8)?
\leq{\sf A}_{9}
&
\verb?AGL(2, 3)?
\leq{\sf S}_{9}
&
\verb?A(5)?
\leq{\sf A}_{10}
&
\verb?M(10)?
\leq{\sf A}_{10}
\\
\verb?S(5)?
\leq{\sf S}_{10}
&
\verb?PGammaL(2, 9)?
\leq{\sf S}_{10}
&
\verb?11:5?
\leq{\sf A}_{11}
&
\verb?M(11)?
\leq{\sf A}_{11}
\\
\verb?AGL(1, 11)?
\leq{\sf S}_{11}
&
\verb?M(11)?
\leq{\sf A}_{12}
&
\verb?M(12)?
\leq{\sf A}_{12}
&
\verb?PSL(2, 11)?
\leq{\sf A}_{12}
\\
\verb?PGL(2, 11)?
\leq{\sf S}_{12}
&
\verb?13:6?
\leq{\sf A}_{13}
&
\verb?L(3, 3)?
\leq{\sf A}_{13}
&
\verb?AGL(1, 13)?
\leq{\sf S}_{13}
\\
\verb?PSL(2,13)?
\leq{\sf A}_{14}
&
\verb?PGL(2,13)?
\leq{\sf S}_{14}
&
\verb?PSL(4, 2)?
\leq{\sf A}_{15}
&
\verb?2^4.PSL(4, 2)?
\leq{\sf A}_{16}
\\
\verb?17:8?
\leq{\sf A}_{17}
&
\verb?L(2, 2^4):4 = PGammaL(2, 2^4)?
\leq{\sf A}_{17}
&
\verb?AGL(1, 17)?
\leq{\sf S}_{17}
&
\verb?PSL(2,17)?
\leq{\sf A}_{18}
\\
\verb?PGL(2,17)?
\leq{\sf S}_{18}
&
\verb?19:9?
\leq{\sf A}_{19}
&
\verb?AGL(1, 19)?
\leq{\sf S}_{19}
&
\verb?PSL(2,19)?
\leq{\sf A}_{20}
\\
\verb?PGL(2,19)?
\leq{\sf S}_{20}
&
\verb?A(7)?
\leq{\sf A}_{21}
&
\verb?PGL(3, 4)?
\leq{\sf A}_{21}
&
\verb?PGL(2, 7)?
\leq{\sf S}_{21}
\\
\verb?S(7)?
\leq{\sf S}_{21}
&
\verb?PGammaL(3, 4)?
\leq{\sf S}_{21}
&
\verb?M(22)?
\leq{\sf A}_{22}
&
\verb?M(22):2?
\leq{\sf S}_{22}
\\
\verb?23:11?
\leq{\sf A}_{23}
&
\verb?M(23)?
\leq{\sf A}_{23}
&
\verb?AGL(1, 23)?
\leq{\sf S}_{23}
&
\verb?M(24)?
\leq{\sf A}_{24}
\\
\verb?PSL(2, 23)?
\leq{\sf A}_{24}
&
\verb?PGL(2, 23)?
\leq{\sf S}_{24}
&
\verb?ASL(2, 5):2?
\leq{\sf A}_{25}
&
\verb?(A(5) x A(5)):2^2?
\leq{\sf A}_{25}
\\
\verb?AGL(2, 5)?
\leq{\sf S}_{25}
&
\verb?(S(5) x S(5)):2?
\leq{\sf S}_{25}
&
\verb?PSigmaL(2, 25)?
\leq{\sf A}_{26}
&
\verb?PGammaL(2, 25)?
\leq{\sf S}_{26}
\\
\verb?ASL(3, 3)?
\leq{\sf A}_{27}
&
\verb?PSp(4, 3):2?
\leq{\sf A}_{27}
&
\verb?AGL(3, 3)?
\leq{\sf S}_{27}
&
\verb?PGammaL(2, 8)?
\leq{\sf A}_{28}
\\
\verb?PGammaU(3, 3)?
\leq{\sf A}_{28}
&
\verb?PSp(6, 2)?
\leq{\sf A}_{28}
&
\verb?S(8)?
\leq{\sf A}_{28}
&
\verb?PSL(2, 27):3?
\leq{\sf A}_{28}
\\
\verb?PGammaL(2, 27)?
\leq{\sf S}_{28}
&
\verb?29:14?
\leq{\sf A}_{29}
&
\verb?AGL(1, 29)?
\leq{\sf S}_{29}
&
\verb?PSL(2,29)?
\leq{\sf A}_{30}
\\
\verb?PGL(2,29)?
\leq{\sf S}_{30}
&
\verb?L(3, 5)?
\leq{\sf A}_{31}
&
\verb?L(5, 2)?
\leq{\sf A}_{31}
&
\verb?AGL(1, 31)?
\leq{\sf S}_{31}
\\
\verb?ASL(5, 2)?
\leq{\sf A}_{32}
&
\verb?PSL(2, 31)?
\leq{\sf A}_{32}
&
\verb?PGL(2, 31)?
\leq{\sf S}_{32}
&
\verb?PGammaL(2, 32)?
\leq{\sf A}_{33}
\\
\verb?S(8)?
\leq{\sf A}_{35}
&
\verb?PGammaU(3, 3)?
\leq{\sf A}_{36}
&
\verb?PSp(4, 3):2?
\leq{\sf A}_{36}
&
\verb?PSp(6, 2)?
\leq{\sf A}_{36}
\\
\verb?A(9)?
\leq{\sf A}_{36}
&
\verb?S(9)?
\leq{\sf S}_{36}
&
\verb?(S(6) x S(6)):2?
\leq{\sf S}_{36}
&
\verb?(S(5) x S(5)):2?
\leq{\sf S}_{36}
\\
\verb?PSL(2,37)?
\leq{\sf A}_{38}
&
\verb?PGL(2,37)?
\leq{\sf S}_{38}
&
\verb?PSp(4, 3):2?
\leq{\sf A}_{40}
&
\verb?PSp(4, 3)?
\leq{\sf A}_{40}
\\
\verb?PSL(4, 3)?
\leq{\sf A}_{40}
&
\verb?PSp(4, 3):2?
\leq{\sf S}_{40}
&
\verb?PGL(4, 3)?
\leq{\sf S}_{40}
&
\verb?PSL(2,41)?
\leq{\sf A}_{42}
\\
\verb?PGL(2,41)?
\leq{\sf S}_{42}
&
\verb?PSL(2,43)?
\leq{\sf A}_{44}
&
\verb?PGL(2,43)?
\leq{\sf S}_{44}
&
\verb?S(10)?
\leq{\sf A}_{45}
\\
\verb?PSp(4, 3):2?
\leq{\sf S}_{45}
&
\verb?(A(7) x A(7)):4?
\leq{\sf A}_{49}
&
\verb?(S(7) x S(7)):2?
\leq{\sf S}_{49}
&
\verb?PSU(3, 5):2?
\leq{\sf A}_{50}
\\
\verb?Alt(11)?
\leq{\sf A}_{55}
&
\verb?Sym(11)?
\leq{\sf S}_{55}
&
\verb?PSL(3, 7).3?
\leq{\sf A}_{57}
&
\verb?PSL(6, 2)?
\leq{\sf A}_{63}
\\
\verb?AGL(6, 2)?
\leq{\sf A}_{64}
&
\verb?Sym(8) wreath Sym(2)?
\leq{\sf A}_{64}
&
\verb?Sym(12)?
\leq{\sf A}_{66}
&
\verb?Sym(13)?
\leq{\sf S}_{78}
\\
\verb?Alt(9)^2.2^2?
\leq{\sf A}_{81}
&
\verb?Sym(9) wreath Sym(2)?
\leq{\sf S}_{81}
&
\verb?Sym(10) wreath Sym(2)?
\leq{\sf S}_{100}
\\
\hline
\end{array}
$$}

\newpage

\section{Table 2}\label{app2}

The following table contains the values of $u$ and $a(u)$ 
where $u=|G:H|$, $G$ and $H$ are as in Appendix~\ref{app1}, and $\deg G=\deg
H\geq 14$. One can read off that in each of the cases $a(u)^3>u$, 
which implies that $G$ cannot be an automorphism group of a flag-transitive
generalized hexagon with point stabilizer $H$. See the proof of Theorem~\ref{mainas}.

\bigskip

{\tiny$$\begin{array}{cccc}\begin{array}{c}
\verb?PSL(2,13)?
\leq{\sf A}_{14}\\
a(u)^3\approx6.86\cdot 10^{18}\\
u\approx3.99\cdot 10^{7}\end{array}
&
\begin{array}{c}
\verb?PGL(2,13)?
\leq{\sf S}_{14}\\
a(u)^3\approx6.86\cdot 10^{18}\\
u\approx3.99\cdot 10^{7}\end{array}
&
\begin{array}{c}
\verb?PSL(4, 2)?
\leq{\sf A}_{15}\\
a(u)^3\approx1.67\cdot 10^{15}\\
u\approx3.24\cdot 10^{7}\end{array}
&
\begin{array}{c}
\verb?2^4.PSL(4, 2)?
\leq{\sf A}_{16}\\
a(u)^3\approx1.67\cdot 10^{15}\\
u\approx3.24\cdot 10^{7}\end{array}
\\\\\hline\\
\begin{array}{c}
\verb?17:8?
\leq{\sf A}_{17}\\
a(u)^3\approx3.20\cdot 10^{26}\\
u\approx1.30\cdot 10^{12}\end{array}
&
\begin{array}{c}
\verb?L(2, 2^4):4 = PGammaL(2, 2^4)?
\leq{\sf A}_{17}\\
a(u)^3\approx1.85\cdot 10^{20}\\
u\approx1.08\cdot 10^{10}\end{array}
&
\begin{array}{c}
\verb?AGL(1, 17)?
\leq{\sf S}_{17}\\
a(u)^3\approx3.20\cdot 10^{26}\\
u\approx1.30\cdot 10^{12}\end{array}
&
\begin{array}{c}
\verb?PSL(2,17)?
\leq{\sf A}_{18}\\
a(u)^3\approx3.20\cdot 10^{26}\\
u\approx1.30\cdot 10^{12}\end{array}
\\\\\hline\\
\begin{array}{c}
\verb?PGL(2,17)?
\leq{\sf S}_{18}\\
a(u)^3\approx3.20\cdot 10^{26}\\
u\approx1.30\cdot 10^{12}\end{array}
&
\begin{array}{c}
\verb?19:9?
\leq{\sf A}_{19}\\
a(u)^3\approx6.44\cdot 10^{33}\\
u\approx3.55\cdot 10^{14}\end{array}
&
\begin{array}{c}
\verb?AGL(1, 19)?
\leq{\sf S}_{19}\\
a(u)^3\approx6.44\cdot 10^{33}\\
u\approx3.55\cdot 10^{14}\end{array}
&
\begin{array}{c}
\verb?PSL(2,19)?
\leq{\sf A}_{20}\\
a(u)^3\approx6.44\cdot 10^{33}\\
u\approx3.55\cdot 10^{14}\end{array}
\\\\\hline\\
\begin{array}{c}
\verb?PGL(2,19)?
\leq{\sf S}_{20}\\
a(u)^3\approx6.44\cdot 10^{33}\\
u\approx3.55\cdot 10^{14}\end{array}
&
\begin{array}{c}
\verb?A(7)?
\leq{\sf A}_{21}\\
a(u)^3\approx2.17\cdot 10^{34}\\
u\approx1.01\cdot 10^{16}\end{array}
&
\begin{array}{c}
\verb?PGL(3, 4)?
\leq{\sf A}_{21}\\
a(u)^3\approx1.57\cdot 10^{30}\\
u\approx4.22\cdot 10^{14}\end{array}
&
\begin{array}{c}
\verb?PGL(2, 7)?
\leq{\sf S}_{21}\\
a(u)^3\approx7.34\cdot 10^{37}\\
u\approx1.52\cdot 10^{17}\end{array}
\\\\\hline\\
\begin{array}{c}
\verb?S(7)?
\leq{\sf S}_{21}\\
a(u)^3\approx2.17\cdot 10^{34}\\
u\approx1.01\cdot 10^{16}\end{array}
&
\begin{array}{c}
\verb?PGammaL(3, 4)?
\leq{\sf S}_{21}\\
a(u)^3\approx1.57\cdot 10^{30}\\
u\approx4.22\cdot 10^{14}\end{array}
&
\begin{array}{c}
\verb?M(22)?
\leq{\sf A}_{22}\\
a(u)^3\approx4.25\cdot 10^{31}\\
u\approx1.26\cdot 10^{15}\end{array}
&
\begin{array}{c}
\verb?M(22):2?
\leq{\sf S}_{22}\\
a(u)^3\approx4.25\cdot 10^{31}\\
u\approx1.26\cdot 10^{15}\end{array}
\\\\\hline\\
\begin{array}{c}
\verb?23:11?
\leq{\sf A}_{23}\\
a(u)^3\approx8.12\cdot 10^{42}\\
u\approx5.10\cdot 10^{19}\end{array}
&
\begin{array}{c}
\verb?M(23)?
\leq{\sf A}_{23}\\
a(u)^3\approx4.25\cdot 10^{31}\\
u\approx1.26\cdot 10^{15}\end{array}
&
\begin{array}{c}
\verb?AGL(1, 23)?
\leq{\sf S}_{23}\\
a(u)^3\approx8.12\cdot 10^{42}\\
u\approx5.10\cdot 10^{19}\end{array}
&
\begin{array}{c}
\verb?M(24)?
\leq{\sf A}_{24}\\
a(u)^3\approx4.25\cdot 10^{31}\\
u\approx1.26\cdot 10^{15}\end{array}
\\\\\hline\\
\begin{array}{c}
\verb?PSL(2, 23)?
\leq{\sf A}_{24}\\
a(u)^3\approx8.12\cdot 10^{42}\\
u\approx5.10\cdot 10^{19}\end{array}
&
\begin{array}{c}
\verb?PGL(2, 23)?
\leq{\sf S}_{24}\\
a(u)^3\approx8.12\cdot 10^{42}\\
u\approx5.10\cdot 10^{19}\end{array}
&
\begin{array}{c}
\verb?ASL(2, 5):2?
\leq{\sf A}_{25}\\
a(u)^3\approx1.31\cdot 10^{47}\\
u\approx1.29\cdot 10^{21}\end{array}
&
\begin{array}{c}
\verb?(A(5) x A(5)):2^2?
\leq{\sf A}_{25}\\
a(u)^3\approx9.51\cdot 10^{45}\\
u\approx5.38\cdot 10^{20}\end{array}
\\\\\hline\\
\begin{array}{c}
\verb?AGL(2, 5)?
\leq{\sf S}_{25}\\
a(u)^3\approx1.31\cdot 10^{47}\\
u\approx1.29\cdot 10^{21}\end{array}
&
\begin{array}{c}
\verb?(S(5) x S(5)):2?
\leq{\sf S}_{25}\\
a(u)^3\approx9.51\cdot 10^{45}\\
u\approx5.38\cdot 10^{20}\end{array}
&
\begin{array}{c}
\verb?PSigmaL(2, 25)?
\leq{\sf A}_{26}\\
a(u)^3\approx1.31\cdot 10^{50}\\
u\approx1.29\cdot 10^{22}\end{array}
&
\begin{array}{c}
\verb?PGammaL(2, 25)?
\leq{\sf S}_{26}\\
a(u)^3\approx1.31\cdot 10^{50}\\
u\approx1.29\cdot 10^{22}\end{array}
\\\\\hline\\
\begin{array}{c}
\verb?ASL(3, 3)?
\leq{\sf A}_{27}\\
a(u)^3\approx2.81\cdot 10^{51}\\
u\approx3.59\cdot 10^{22}\end{array}
&
\begin{array}{c}
\verb?PSp(4, 3):2?
\leq{\sf A}_{27}\\
a(u)^3\approx3.21\cdot 10^{49}\\
u\approx1.05\cdot 10^{23}\end{array}
&
\begin{array}{c}
\verb?AGL(3, 3)?
\leq{\sf S}_{27}\\
a(u)^3\approx2.81\cdot 10^{51}\\
u\approx3.59\cdot 10^{22}\end{array}
&
\begin{array}{c}
\verb?PGammaL(2, 8)?
\leq{\sf A}_{28}\\
a(u)^3\approx2.84\cdot 10^{58}\\
u\approx1.00\cdot 10^{26}\end{array}
\\\\\hline\\
\begin{array}{c}
\verb?PGammaU(3, 3)?
\leq{\sf A}_{28}\\
a(u)^3\approx5.54\cdot 10^{55}\\
u\approx1.26\cdot 10^{25}\end{array}
&
\begin{array}{c}
\verb?PSp(6, 2)?
\leq{\sf A}_{28}\\
a(u)^3\approx3.21\cdot 10^{49}\\
u\approx1.05\cdot 10^{23}\end{array}
&
\begin{array}{c}
\verb?S(8)?
\leq{\sf A}_{28}\\
a(u)^3\approx1.49\cdot 10^{54}\\
u\approx3.78\cdot 10^{24}\end{array}
&
\begin{array}{c}
\verb?PSL(2, 27):3?
\leq{\sf A}_{28}\\
a(u)^3\approx8.41\cdot 10^{57}\\
u\approx5.17\cdot 10^{24}\end{array}
\end{array}$$}{\tiny$$\begin{array}{cccc}\begin{array}{c}
\verb?PGammaL(2, 27)?
\leq{\sf S}_{28}\\
a(u)^3\approx8.41\cdot 10^{57}\\
u\approx5.17\cdot 10^{24}\end{array}
&
\begin{array}{c}
\verb?29:14?
\leq{\sf A}_{29}\\
a(u)^3\approx3.57\cdot 10^{64}\\
u\approx1.08\cdot 10^{28}\end{array}
&
\begin{array}{c}
\verb?AGL(1, 29)?
\leq{\sf S}_{29}\\
a(u)^3\approx3.57\cdot 10^{64}\\
u\approx1.08\cdot 10^{28}\end{array}
&
\begin{array}{c}
\verb?PSL(2,29)?
\leq{\sf A}_{30}\\
a(u)^3\approx3.57\cdot 10^{64}\\
u\approx1.08\cdot 10^{28}\end{array}
\\\\\hline\\
\begin{array}{c}
\verb?PGL(2,29)?
\leq{\sf S}_{30}\\
a(u)^3\approx3.57\cdot 10^{64}\\
u\approx1.08\cdot 10^{28}\end{array}
&
\begin{array}{c}
\verb?L(3, 5)?
\leq{\sf A}_{31}\\
a(u)^3\approx1.09\cdot 10^{62}\\
u\approx1.10\cdot 10^{28}\end{array}
&
\begin{array}{c}
\verb?L(5, 2)?
\leq{\sf A}_{31}\\
a(u)^3\approx1.92\cdot 10^{60}\\
u\approx4.11\cdot 10^{26}\end{array}
&
\begin{array}{c}
\verb?AGL(1, 31)?
\leq{\sf S}_{31}\\
a(u)^3\approx5.58\cdot 10^{70}\\
u\approx8.84\cdot 10^{30}\end{array}
\\\\\hline\\
\begin{array}{c}
\verb?ASL(5, 2)?
\leq{\sf A}_{32}\\
a(u)^3\approx1.92\cdot 10^{60}\\
u\approx4.11\cdot 10^{26}\end{array}
&
\begin{array}{c}
\verb?PSL(2, 31)?
\leq{\sf A}_{32}\\
a(u)^3\approx5.58\cdot 10^{70}\\
u\approx8.84\cdot 10^{30}\end{array}
&
\begin{array}{c}
\verb?PGL(2, 31)?
\leq{\sf S}_{32}\\
a(u)^3\approx5.58\cdot 10^{70}\\
u\approx8.84\cdot 10^{30}\end{array}
&
\begin{array}{c}
\verb?PGammaL(2, 32)?
\leq{\sf A}_{33}\\
a(u)^3\approx1.50\cdot 10^{72}\\
u\approx2.65\cdot 10^{31}\end{array}
\\\\\hline\\
\begin{array}{c}
\verb?S(8)?
\leq{\sf A}_{35}\\
a(u)^3\approx5.70\cdot 10^{78}\\
u\approx1.28\cdot 10^{35}\end{array}
&
\begin{array}{c}
\verb?PGammaU(3, 3)?
\leq{\sf A}_{36}\\
a(u)^3\approx9.86\cdot 10^{84}\\
u\approx1.53\cdot 10^{37}\end{array}
&
\begin{array}{c}
\verb?PSp(4, 3):2?
\leq{\sf A}_{36}\\
a(u)^3\approx3.65\cdot 10^{80}\\
u\approx3.58\cdot 10^{36}\end{array}
&
\begin{array}{c}
\verb?PSp(6, 2)?
\leq{\sf A}_{36}\\
a(u)^3\approx5.70\cdot 10^{78}\\
u\approx1.28\cdot 10^{35}\end{array}
\\\\\hline\\
\begin{array}{c}
\verb?A(9)?
\leq{\sf A}_{36}\\
a(u)^3\approx2.92\cdot 10^{81}\\
u\approx1.02\cdot 10^{36}\end{array}
&
\begin{array}{c}
\verb?S(9)?
\leq{\sf S}_{36}\\
a(u)^3\approx2.92\cdot 10^{81}\\
u\approx1.02\cdot 10^{36}\end{array}
&
\begin{array}{c}
\verb?(S(6) x S(6)):2?
\leq{\sf S}_{36}\\
a(u)^3\approx3.65\cdot 10^{77}\\
u\approx3.58\cdot 10^{35}\end{array}
&
\begin{array}{c}
\verb?(S(5) x S(5)):2?
\leq{\sf S}_{36}\\
a(u)^3\approx1.70\cdot 10^{82}\\
u\approx1.29\cdot 10^{37}\end{array}
\\\\\hline\\
\begin{array}{c}
\verb?PSL(2,37)?
\leq{\sf A}_{38}\\
a(u)^3\approx8.72\cdot 10^{90}\\
u\approx1.03\cdot 10^{40}\end{array}
&
\begin{array}{c}
\verb?PGL(2,37)?
\leq{\sf S}_{38}\\
a(u)^3\approx8.72\cdot 10^{90}\\
u\approx1.03\cdot 10^{40}\end{array}
&
\begin{array}{c}
\verb?PSp(4, 3):2?
\leq{\sf A}_{40}\\
a(u)^3\approx5.05\cdot 10^{87}\\
u\approx7.86\cdot 10^{42}\end{array}
&
\begin{array}{c}
\verb?PSp(4, 3)?
\leq{\sf A}_{40}\\
a(u)^3\approx4.04\cdot 10^{88}\\
u\approx1.57\cdot 10^{43}\end{array}
\\\\\hline\\
\begin{array}{c}
\verb?PSL(4, 3)?
\leq{\sf A}_{40}\\
a(u)^3\approx6.92\cdot 10^{84}\\
u\approx6.72\cdot 10^{40}\end{array}
&
\begin{array}{c}
\verb?PSp(4, 3):2?
\leq{\sf S}_{40}\\
a(u)^3\approx4.04\cdot 10^{88}\\
u\approx1.57\cdot 10^{43}\end{array}
&
\begin{array}{c}
\verb?PGL(4, 3)?
\leq{\sf S}_{40}\\
a(u)^3\approx6.92\cdot 10^{84}\\
u\approx6.72\cdot 10^{40}\end{array}
&
\begin{array}{c}
\verb?PSL(2,41)?
\leq{\sf A}_{42}\\
a(u)^3\approx8.79\cdot 10^{97}\\
u\approx2.03\cdot 10^{46}\end{array}
\\\\\hline\\
\begin{array}{c}
\verb?PGL(2,41)?
\leq{\sf S}_{42}\\
a(u)^3\approx8.79\cdot 10^{97}\\
u\approx2.03\cdot 10^{46}\end{array}
&
\begin{array}{c}
\verb?PSL(2,43)?
\leq{\sf A}_{44}\\
a(u)^3\approx3.87\cdot 10^{107}\\
u\approx3.34\cdot 10^{49}\end{array}
&
\begin{array}{c}
\verb?PGL(2,43)?
\leq{\sf S}_{44}\\
a(u)^3\approx3.87\cdot 10^{107}\\
u\approx3.34\cdot 10^{49}\end{array}
&
\begin{array}{c}
\verb?S(10)?
\leq{\sf A}_{45}\\
a(u)^3\approx5.83\cdot 10^{101}\\
u\approx1.64\cdot 10^{49}\end{array}
\\\\\hline\\
\begin{array}{c}
\verb?PSp(4, 3):2?
\leq{\sf S}_{45}\\
a(u)^3\approx4.66\cdot 10^{105}\\
u\approx2.30\cdot 10^{51}\end{array}
&
\begin{array}{c}
\verb?(A(7) x A(7)):4?
\leq{\sf A}_{49}\\
a(u)^3\approx6.52\cdot 10^{116}\\
u\approx1.19\cdot 10^{55}\end{array}
&
\begin{array}{c}
\verb?(S(7) x S(7)):2?
\leq{\sf S}_{49}\\
a(u)^3\approx6.52\cdot 10^{116}\\
u\approx1.19\cdot 10^{55}\end{array}
&
\begin{array}{c}
\verb?PSU(3, 5):2?
\leq{\sf A}_{50}\\
a(u)^3\approx2.43\cdot 10^{125}\\
u\approx6.03\cdot 10^{58}\end{array}
\\\\\hline\\
\begin{array}{c}
\verb?Alt(11)?
\leq{\sf A}_{55}\\
a(u)^3\approx1.62\cdot 10^{142}\\
u\approx3.18\cdot 10^{65}\end{array}
&
\begin{array}{c}
\verb?Sym(11)?
\leq{\sf S}_{55}\\
a(u)^3\approx1.62\cdot 10^{142}\\
u\approx3.18\cdot 10^{65}\end{array}
&
\begin{array}{c}
\verb?PSL(3, 7).3?
\leq{\sf A}_{57}\\
a(u)^3\approx8.06\cdot 10^{156}\\
u\approx3.59\cdot 10^{69}\end{array}
&
\begin{array}{c}
\verb?PSL(6, 2)?
\leq{\sf A}_{63}\\
a(u)^3\approx1.12\cdot 10^{164}\\
u\approx4.91\cdot 10^{76}\end{array}
\end{array}$$}{\tiny$$\begin{array}{cccc}\begin{array}{c}
\verb?AGL(6, 2)?
\leq{\sf A}_{64}\\
a(u)^3\approx1.12\cdot 10^{164}\\
u\approx4.91\cdot 10^{76}\end{array}
&
\begin{array}{c}
\verb?Sym(8) wreath Sym(2)?
\leq{\sf A}_{64}\\
a(u)^3\approx2.35\cdot 10^{167}\\
u\approx1.95\cdot 10^{79}\end{array}
&
\begin{array}{c}
\verb?Sym(12)?
\leq{\sf A}_{66}\\
a(u)^3\approx7.71\cdot 10^{174}\\
u\approx5.68\cdot 10^{83}\end{array}
&
\begin{array}{c}
\verb?Sym(13)?
\leq{\sf S}_{78}\\
a(u)^3\approx5.28\cdot 10^{214}\\
u\approx1.81\cdot 10^{105}\end{array}
\\\\\hline\\
\begin{array}{c}
\verb?Alt(9)^2.2^2?
\leq{\sf A}_{81}\\
a(u)^3\approx2.97\cdot 10^{220}\\
u\approx2.20\cdot 10^{109}\end{array}
&
\begin{array}{c}
\verb?Sym(9) wreath Sym(2)?
\leq{\sf S}_{81}\\
a(u)^3\approx2.97\cdot 10^{220}\\
u\approx2.20\cdot 10^{109}\end{array}
&
\begin{array}{c}
\verb?Sym(10) wreath Sym(2)?
\leq{\sf S}_{100}\\
a(u)^3\approx2.74\cdot 10^{293}\\
u\approx3.54\cdot 10^{144}\end{array}
&
\end{array}$$}

\newpage

\section{Table 3}\label{app3}

The following table contains the values of $u$ and $b(u)$,
where $u=|G:H|$, $G$ and $H$ are as in Appendix~\ref{app1}, and $\deg G=\deg
H\geq 14$. One can read off that in each of the cases $b(u)^2>u$, 
which implies that $G$ cannot be a automorphism group of a flag-transitive
generalized 
octagon with point stabilizer $H$. See the proof of Theorem~\ref{mainas}.

\bigskip

{\tiny$$\begin{array}{cccc}\begin{array}{c}
\verb?PSL(2,13)?
\leq{\sf A}_{14}\\
b(u)^2\approx3.61\cdot 10^{12}\\
u\approx3.99\cdot 10^{7}\end{array}
&
\begin{array}{c}
\verb?PGL(2,13)?
\leq{\sf S}_{14}\\
b(u)^2\approx3.61\cdot 10^{12}\\
u\approx3.99\cdot 10^{7}\end{array}
&
\begin{array}{c}
\verb?PSL(4, 2)?
\leq{\sf A}_{15}\\
b(u)^2\approx1.41\cdot 10^{10}\\
u\approx3.24\cdot 10^{7}\end{array}
&
\begin{array}{c}
\verb?2^4.PSL(4, 2)?
\leq{\sf A}_{16}\\
b(u)^2\approx1.41\cdot 10^{10}\\
u\approx3.24\cdot 10^{7}\end{array}
\\\\\hline\\
\begin{array}{c}
\verb?17:8?
\leq{\sf A}_{17}\\
b(u)^2\approx4.68\cdot 10^{17}\\
u\approx1.30\cdot 10^{12}\end{array}
&
\begin{array}{c}
\verb?L(2, 2^4):4 = PGammaL(2, 2^4)?
\leq{\sf A}_{17}\\
b(u)^2\approx3.25\cdot 10^{13}\\
u\approx1.08\cdot 10^{10}\end{array}
&
\begin{array}{c}
\verb?AGL(1, 17)?
\leq{\sf S}_{17}\\
b(u)^2\approx4.68\cdot 10^{17}\\
u\approx1.30\cdot 10^{12}\end{array}
&
\begin{array}{c}
\verb?PSL(2,17)?
\leq{\sf A}_{18}\\
b(u)^2\approx4.68\cdot 10^{17}\\
u\approx1.30\cdot 10^{12}\end{array}
\\\\\hline\\
\begin{array}{c}
\verb?PGL(2,17)?
\leq{\sf S}_{18}\\
b(u)^2\approx4.68\cdot 10^{17}\\
u\approx1.30\cdot 10^{12}\end{array}
&
\begin{array}{c}
\verb?19:9?
\leq{\sf A}_{19}\\
b(u)^2\approx3.46\cdot 10^{22}\\
u\approx3.55\cdot 10^{14}\end{array}
&
\begin{array}{c}
\verb?AGL(1, 19)?
\leq{\sf S}_{19}\\
b(u)^2\approx3.46\cdot 10^{22}\\
u\approx3.55\cdot 10^{14}\end{array}
&
\begin{array}{c}
\verb?PSL(2,19)?
\leq{\sf A}_{20}\\
b(u)^2\approx3.46\cdot 10^{22}\\
u\approx3.55\cdot 10^{14}\end{array}
\\\\\hline\\
\begin{array}{c}
\verb?PGL(2,19)?
\leq{\sf S}_{20}\\
b(u)^2\approx3.46\cdot 10^{22}\\
u\approx3.55\cdot 10^{14}\end{array}
&
\begin{array}{c}
\verb?A(7)?
\leq{\sf A}_{21}\\
b(u)^2\approx7.79\cdot 10^{22}\\
u\approx1.01\cdot 10^{16}\end{array}
&
\begin{array}{c}
\verb?PGL(3, 4)?
\leq{\sf A}_{21}\\
b(u)^2\approx1.35\cdot 10^{20}\\
u\approx4.22\cdot 10^{14}\end{array}
&
\begin{array}{c}
\verb?PGL(2, 7)?
\leq{\sf S}_{21}\\
b(u)^2\approx1.75\cdot 10^{25}\\
u\approx1.52\cdot 10^{17}\end{array}
\\\\\hline\\
\begin{array}{c}
\verb?S(7)?
\leq{\sf S}_{21}\\
b(u)^2\approx7.79\cdot 10^{22}\\
u\approx1.01\cdot 10^{16}\end{array}
&
\begin{array}{c}
\verb?PGammaL(3, 4)?
\leq{\sf S}_{21}\\
b(u)^2\approx1.35\cdot 10^{20}\\
u\approx4.22\cdot 10^{14}\end{array}
&
\begin{array}{c}
\verb?M(22)?
\leq{\sf A}_{22}\\
b(u)^2\approx1.21\cdot 10^{21}\\
u\approx1.26\cdot 10^{15}\end{array}
&
\begin{array}{c}
\verb?M(22):2?
\leq{\sf S}_{22}\\
b(u)^2\approx1.21\cdot 10^{21}\\
u\approx1.26\cdot 10^{15}\end{array}
\\\\\hline\\
\begin{array}{c}
\verb?23:11?
\leq{\sf A}_{23}\\
b(u)^2\approx4.04\cdot 10^{28}\\
u\approx5.10\cdot 10^{19}\end{array}
&
\begin{array}{c}
\verb?M(23)?
\leq{\sf A}_{23}\\
b(u)^2\approx1.21\cdot 10^{21}\\
u\approx1.26\cdot 10^{15}\end{array}
&
\begin{array}{c}
\verb?AGL(1, 23)?
\leq{\sf S}_{23}\\
b(u)^2\approx4.04\cdot 10^{28}\\
u\approx5.10\cdot 10^{19}\end{array}
&
\begin{array}{c}
\verb?M(24)?
\leq{\sf A}_{24}\\
b(u)^2\approx1.21\cdot 10^{21}\\
u\approx1.26\cdot 10^{15}\end{array}
\\\\\hline\\
\begin{array}{c}
\verb?PSL(2, 23)?
\leq{\sf A}_{24}\\
b(u)^2\approx4.04\cdot 10^{28}\\
u\approx5.10\cdot 10^{19}\end{array}
&
\begin{array}{c}
\verb?PGL(2, 23)?
\leq{\sf S}_{24}\\
b(u)^2\approx4.04\cdot 10^{28}\\
u\approx5.10\cdot 10^{19}\end{array}
&
\begin{array}{c}
\verb?ASL(2, 5):2?
\leq{\sf A}_{25}\\
b(u)^2\approx2.58\cdot 10^{31}\\
u\approx1.29\cdot 10^{21}\end{array}
&
\begin{array}{c}
\verb?(A(5) x A(5)):2^2?
\leq{\sf A}_{25}\\
b(u)^2\approx4.49\cdot 10^{30}\\
u\approx5.38\cdot 10^{20}\end{array}
\\\\\hline\\
\begin{array}{c}
\verb?AGL(2, 5)?
\leq{\sf S}_{25}\\
b(u)^2\approx2.58\cdot 10^{31}\\
u\approx1.29\cdot 10^{21}\end{array}
&
\begin{array}{c}
\verb?(S(5) x S(5)):2?
\leq{\sf S}_{25}\\
b(u)^2\approx4.49\cdot 10^{30}\\
u\approx5.38\cdot 10^{20}\end{array}
&
\begin{array}{c}
\verb?PSigmaL(2, 25)?
\leq{\sf A}_{26}\\
b(u)^2\approx2.58\cdot 10^{33}\\
u\approx1.29\cdot 10^{22}\end{array}
&
\begin{array}{c}
\verb?PGammaL(2, 25)?
\leq{\sf S}_{26}\\
b(u)^2\approx2.58\cdot 10^{33}\\
u\approx1.29\cdot 10^{22}\end{array}
\\\\\hline\\
\begin{array}{c}
\verb?ASL(3, 3)?
\leq{\sf A}_{27}\\
b(u)^2\approx1.99\cdot 10^{34}\\
u\approx3.59\cdot 10^{22}\end{array}
&
\begin{array}{c}
\verb?PSp(4, 3):2?
\leq{\sf A}_{27}\\
b(u)^2\approx1.01\cdot 10^{33}\\
u\approx1.05\cdot 10^{23}\end{array}
&
\begin{array}{c}
\verb?AGL(3, 3)?
\leq{\sf S}_{27}\\
b(u)^2\approx1.99\cdot 10^{34}\\
u\approx3.59\cdot 10^{22}\end{array}
&
\begin{array}{c}
\verb?PGammaL(2, 8)?
\leq{\sf A}_{28}\\
b(u)^2\approx9.31\cdot 10^{38}\\
u\approx1.00\cdot 10^{26}\end{array}
\\\\\hline\\
\begin{array}{c}
\verb?PGammaU(3, 3)?
\leq{\sf A}_{28}\\
b(u)^2\approx1.45\cdot 10^{37}\\
u\approx1.26\cdot 10^{25}\end{array}
&
\begin{array}{c}
\verb?PSp(6, 2)?
\leq{\sf A}_{28}\\
b(u)^2\approx1.01\cdot 10^{33}\\
u\approx1.05\cdot 10^{23}\end{array}
&
\begin{array}{c}
\verb?S(8)?
\leq{\sf A}_{28}\\
b(u)^2\approx1.30\cdot 10^{36}\\
u\approx3.78\cdot 10^{24}\end{array}
&
\begin{array}{c}
\verb?PSL(2, 27):3?
\leq{\sf A}_{28}\\
b(u)^2\approx4.13\cdot 10^{38}\\
u\approx5.17\cdot 10^{24}\end{array}
\end{array}$$}{\tiny $$\begin{array}{cccc}\begin{array}{c}
\verb?PGammaL(2, 27)?
\leq{\sf S}_{28}\\
b(u)^2\approx4.13\cdot 10^{38}\\
u\approx5.17\cdot 10^{24}\end{array}
&
\begin{array}{c}
\verb?29:14?
\leq{\sf A}_{29}\\
b(u)^2\approx1.08\cdot 10^{43}\\
u\approx1.08\cdot 10^{28}\end{array}
&
\begin{array}{c}
\verb?AGL(1, 29)?
\leq{\sf S}_{29}\\
b(u)^2\approx1.08\cdot 10^{43}\\
u\approx1.08\cdot 10^{28}\end{array}
&
\begin{array}{c}
\verb?PSL(2,29)?
\leq{\sf A}_{30}\\
b(u)^2\approx1.08\cdot 10^{43}\\
u\approx1.08\cdot 10^{28}\end{array}
\\\\\hline\\
\begin{array}{c}
\verb?PGL(2,29)?
\leq{\sf S}_{30}\\
b(u)^2\approx1.08\cdot 10^{43}\\
u\approx1.08\cdot 10^{28}\end{array}
&
\begin{array}{c}
\verb?L(3, 5)?
\leq{\sf A}_{31}\\
b(u)^2\approx2.28\cdot 10^{41}\\
u\approx1.10\cdot 10^{28}\end{array}
&
\begin{array}{c}
\verb?L(5, 2)?
\leq{\sf A}_{31}\\
b(u)^2\approx1.54\cdot 10^{40}\\
u\approx4.11\cdot 10^{26}\end{array}
&
\begin{array}{c}
\verb?AGL(1, 31)?
\leq{\sf S}_{31}\\
b(u)^2\approx1.46\cdot 10^{47}\\
u\approx8.84\cdot 10^{30}\end{array}
\\\\\hline\\
\begin{array}{c}
\verb?ASL(5, 2)?
\leq{\sf A}_{32}\\
b(u)^2\approx1.54\cdot 10^{40}\\
u\approx4.11\cdot 10^{26}\end{array}
&
\begin{array}{c}
\verb?PSL(2, 31)?
\leq{\sf A}_{32}\\
b(u)^2\approx1.46\cdot 10^{47}\\
u\approx8.84\cdot 10^{30}\end{array}
&
\begin{array}{c}
\verb?PGL(2, 31)?
\leq{\sf S}_{32}\\
b(u)^2\approx1.46\cdot 10^{47}\\
u\approx8.84\cdot 10^{30}\end{array}
&
\begin{array}{c}
\verb?PGammaL(2, 32)?
\leq{\sf A}_{33}\\
b(u)^2\approx1.31\cdot 10^{48}\\
u\approx2.65\cdot 10^{31}\end{array}
\\\\\hline\\
\begin{array}{c}
\verb?S(8)?
\leq{\sf A}_{35}\\
b(u)^2\approx3.19\cdot 10^{52}\\
u\approx1.28\cdot 10^{35}\end{array}
&
\begin{array}{c}
\verb?PGammaU(3, 3)?
\leq{\sf A}_{36}\\
b(u)^2\approx4.59\cdot 10^{56}\\
u\approx1.53\cdot 10^{37}\end{array}
&
\begin{array}{c}
\verb?PSp(4, 3):2?
\leq{\sf A}_{36}\\
b(u)^2\approx5.11\cdot 10^{53}\\
u\approx3.58\cdot 10^{36}\end{array}
&
\begin{array}{c}
\verb?PSp(6, 2)?
\leq{\sf A}_{36}\\
b(u)^2\approx3.19\cdot 10^{52}\\
u\approx1.28\cdot 10^{35}\end{array}
\\\\\hline\\
\begin{array}{c}
\verb?A(9)?
\leq{\sf A}_{36}\\
b(u)^2\approx2.04\cdot 10^{54}\\
u\approx1.02\cdot 10^{36}\end{array}
&
\begin{array}{c}
\verb?S(9)?
\leq{\sf S}_{36}\\
b(u)^2\approx2.04\cdot 10^{54}\\
u\approx1.02\cdot 10^{36}\end{array}
&
\begin{array}{c}
\verb?(S(6) x S(6)):2?
\leq{\sf S}_{36}\\
b(u)^2\approx5.11\cdot 10^{51}\\
u\approx3.58\cdot 10^{35}\end{array}
&
\begin{array}{c}
\verb?(S(5) x S(5)):2?
\leq{\sf S}_{36}\\
b(u)^2\approx6.62\cdot 10^{54}\\
u\approx1.29\cdot 10^{37}\end{array}
\\\\\hline\\
\begin{array}{c}
\verb?PSL(2,37)?
\leq{\sf A}_{38}\\
b(u)^2\approx4.23\cdot 10^{60}\\
u\approx1.03\cdot 10^{40}\end{array}
&
\begin{array}{c}
\verb?PGL(2,37)?
\leq{\sf S}_{38}\\
b(u)^2\approx4.23\cdot 10^{60}\\
u\approx1.03\cdot 10^{40}\end{array}
&
\begin{array}{c}
\verb?PSp(4, 3):2?
\leq{\sf A}_{40}\\
b(u)^2\approx2.94\cdot 10^{58}\\
u\approx7.86\cdot 10^{42}\end{array}
&
\begin{array}{c}
\verb?PSp(4, 3)?
\leq{\sf A}_{40}\\
b(u)^2\approx1.17\cdot 10^{59}\\
u\approx1.57\cdot 10^{43}\end{array}
\\\\\hline\\
\begin{array}{c}
\verb?PSL(4, 3)?
\leq{\sf A}_{40}\\
b(u)^2\approx3.63\cdot 10^{56}\\
u\approx6.72\cdot 10^{40}\end{array}
&
\begin{array}{c}
\verb?PSp(4, 3):2?
\leq{\sf S}_{40}\\
b(u)^2\approx1.17\cdot 10^{59}\\
u\approx1.57\cdot 10^{43}\end{array}
&
\begin{array}{c}
\verb?PGL(4, 3)?
\leq{\sf S}_{40}\\
b(u)^2\approx3.63\cdot 10^{56}\\
u\approx6.72\cdot 10^{40}\end{array}
&
\begin{array}{c}
\verb?PSL(2,41)?
\leq{\sf A}_{42}\\
b(u)^2\approx1.97\cdot 10^{65}\\
u\approx2.03\cdot 10^{46}\end{array}
\\\\\hline\\
\begin{array}{c}
\verb?PGL(2,41)?
\leq{\sf S}_{42}\\
b(u)^2\approx1.97\cdot 10^{65}\\
u\approx2.03\cdot 10^{46}\end{array}
&
\begin{array}{c}
\verb?PSL(2,43)?
\leq{\sf A}_{44}\\
b(u)^2\approx5.31\cdot 10^{71}\\
u\approx3.34\cdot 10^{49}\end{array}
&
\begin{array}{c}
\verb?PGL(2,43)?
\leq{\sf S}_{44}\\
b(u)^2\approx5.31\cdot 10^{71}\\
u\approx3.34\cdot 10^{49}\end{array}
&
\begin{array}{c}
\verb?S(10)?
\leq{\sf A}_{45}\\
b(u)^2\approx6.98\cdot 10^{67}\\
u\approx1.64\cdot 10^{49}\end{array}
\\\\\hline\\
\begin{array}{c}
\verb?PSp(4, 3):2?
\leq{\sf S}_{45}\\
b(u)^2\approx2.79\cdot 10^{70}\\
u\approx2.30\cdot 10^{51}\end{array}
&
\begin{array}{c}
\verb?(A(7) x A(7)):4?
\leq{\sf A}_{49}\\
b(u)^2\approx7.52\cdot 10^{77}\\
u\approx1.19\cdot 10^{55}\end{array}
&
\begin{array}{c}
\verb?(S(7) x S(7)):2?
\leq{\sf S}_{49}\\
b(u)^2\approx7.52\cdot 10^{77}\\
u\approx1.19\cdot 10^{55}\end{array}
&
\begin{array}{c}
\verb?PSU(3, 5):2?
\leq{\sf A}_{50}\\
b(u)^2\approx3.89\cdot 10^{83}\\
u\approx6.03\cdot 10^{58}\end{array}
\\\\\hline\\
\begin{array}{c}
\verb?Alt(11)?
\leq{\sf A}_{55}\\
b(u)^2\approx6.40\cdot 10^{94}\\
u\approx3.18\cdot 10^{65}\end{array}
&
\begin{array}{c}
\verb?Sym(11)?
\leq{\sf S}_{55}\\
b(u)^2\approx6.40\cdot 10^{94}\\
u\approx3.18\cdot 10^{65}\end{array}
&
\begin{array}{c}
\verb?PSL(3, 7).3?
\leq{\sf A}_{57}\\
b(u)^2\approx4.02\cdot 10^{104}\\
u\approx3.59\cdot 10^{69}\end{array}
&
\begin{array}{c}
\verb?PSL(6, 2)?
\leq{\sf A}_{63}\\
b(u)^2\approx2.32\cdot 10^{109}\\
u\approx4.91\cdot 10^{76}\end{array}
\end{array}$$}{\tiny $$\begin{array}{cccc}\begin{array}{c}
\verb?AGL(6, 2)?
\leq{\sf A}_{64}\\
b(u)^2\approx2.32\cdot 10^{109}\\
u\approx4.91\cdot 10^{76}\end{array}
&
\begin{array}{c}
\verb?Sym(8) wreath Sym(2)?
\leq{\sf A}_{64}\\
b(u)^2\approx3.81\cdot 10^{111}\\
u\approx1.95\cdot 10^{79}\end{array}
&
\begin{array}{c}
\verb?Sym(12)?
\leq{\sf A}_{66}\\
b(u)^2\approx3.90\cdot 10^{116}\\
u\approx5.68\cdot 10^{83}\end{array}
&
\begin{array}{c}
\verb?Sym(13)?
\leq{\sf S}_{78}\\
b(u)^2\approx1.40\cdot 10^{143}\\
u\approx1.81\cdot 10^{105}\end{array}
\\\\\hline\\
\begin{array}{c}
\verb?Alt(9)^2.2^2?
\leq{\sf A}_{81}\\
b(u)^2\approx9.59\cdot 10^{146}\\
u\approx2.20\cdot 10^{109}\end{array}
&
\begin{array}{c}
\verb?Sym(9) wreath Sym(2)?
\leq{\sf S}_{81}\\
b(u)^2\approx9.59\cdot 10^{146}\\
u\approx2.20\cdot 10^{109}\end{array}
&
\begin{array}{c}
\verb?Sym(10) wreath Sym(2)?
\leq{\sf S}_{100}\\
b(u)^2\approx4.22\cdot 10^{195}\\
u\approx3.54\cdot 10^{144}\end{array}
&
\end{array}
$$}